\documentclass[a4paper]{article}
\usepackage{enumitem}

\usepackage{fullpage}
\usepackage{graphicx} 
\usepackage[hypcap=false]{caption}
\usepackage{subcaption}
\usepackage[section]{placeins}

\usepackage{amssymb} 
\usepackage[intlimits]{amsmath}
\usepackage{amsfonts}
\usepackage{amsthm} 
\usepackage{amsrefs}
     
\usepackage{mathptmx}

\usepackage{hyperref}
\hypersetup{colorlinks=true,linkcolor=RoyalPurple,citecolor=RoyalPurple}

\usepackage{cite}
\allowdisplaybreaks
\usepackage{ulem}
\usepackage{color}
\usepackage[dvipsnames]{xcolor}

\let\oldsqrt\sqrt
\def\sqrt{\mathpalette\DHLhksqrt}
\def\DHLhksqrt#1#2{%
\setbox0=\hbox{$#1\oldsqrt{#2\,}$}\dimen0=\ht0
\advance\dimen0-0.2\ht0
\setbox2=\hbox{\vrule height\ht0 depth -\dimen0}%
{\box0\lower0.4pt\box2}}

\newcommand{\R}{\mathbb{R}} % reelle Zahlen
\newcommand{\N}{\mathbb{N}} % natuerliche Zahlen
\newcommand{\Z}{\mathbb{Z}} % ganze Zahlen
\newcommand{\dist}{\textnormal{dist}} % dist ...
\newcommand{\Ds}{(-\Delta)^s}

\newcommand{\cE}{{\mathcal E}}

\newcommand{\cH}{{\mathcal H}}

\newcommand{\hf}{{\,{}_2F_1}}

\newcommand{\eps}{\varepsilon}

\newtheorem{theorem}{Theorem}[section]
\newtheorem{lemma}[theorem]{Lemma}
\newtheorem{proposition}[theorem]{Proposition}
\newtheorem{corollary}[theorem]{Corollary}

\theoremstyle{remark}
\newtheorem{remark}[theorem]{Remark}

\theoremstyle{definition}

\numberwithin{equation}{section}

\title{Fractional Laplacians on ellipsoids}
\author{
Nicola Abatangelo\footnote{Institut f\"ur Mathematik, Goethe-Universit\"at Frankfurt am Main, Robert-Mayer-Stra\ss e 10, 60325 Frankfurt am Main, Germany. \textit{Email addresses:\ }\ \texttt{abatangelo@math.uni-frankfurt.de} \ and \ \texttt{jarohs@math.uni-frankfurt.de}. 
NA is supported by the Alexander von Humboldt Foundation.\protect\label{goethe}} ,
\ Sven Jarohs\textsuperscript{\normalfont\ref{goethe}}, and
Alberto Salda\~{n}a\footnote{Instituto de Matem\'aticas, Universidad Aut\'onoma de M\'exico, Circuito Exterior, Ciudad Universitaria, 04510 Coyoac\'an, Ciudad de M\'exico, M\'exico. \textit{Email address:}\ \texttt{alberto.saldana@im.unam.mx}.}
}
\date{ }

\begin{document}

\maketitle

{\let\thefootnote\relax
\addtocounter{footnote}{-2}}

\begin{abstract}
We show explicit formulas for the evaluation of (possibly higher-order) fractional Laplacians~$\Ds$ of some functions supported on ellipsoids. In particular, we derive the explicit expression of the torsion function and give examples of~$s$-harmonic functions.  As an application, we infer that the weak maximum principle fails in eccentric ellipsoids for~$s\in(1,\sqrt{3}+3/2)$ in any dimension~$n\geq 2$. We build a counterexample  in terms of the torsion function times a polynomial of degree~$2$. Using point inversion transformations, it follows that a variety of bounded and unbounded domains do not satisfy positivity preserving properties either and we give some examples. 
\end{abstract}

\section{Introduction}

The fractional Laplacian~$(-\Delta)^s$,~$s>0$, is a pseudodifferential operator with Fourier symbol~$|\cdot|^{2s}$ which can be evaluated pointwisely via a hypersingular integral (see~\eqref{Ds:def} below).  This operator has many applications in mathematical modeling and the set of solutions of boundary value problems involving the fractional Laplacian has a rich and complex mathematical structure, see~\cite{fth,ga,bv}.

One of the main obstacles in the study of this operator is the difficulty of evaluating explicitly~$(-\Delta)^s$, even on simple functions, see for example~\cite{dyda,meijer,dcds,ccm} and the references therein for some of the few exceptions that are available in the literature. For the same reason, explicit solutions of boundary value problems are rare. 

In this paper, we show some explicit formulas for the evaluation of the fractional Laplacian of polynomial-like functions supported in ellipsoids. Our first result concerns the explicit expression of the torsion function of an ellipsoid. Let\begin{align*}
 \cH^s_0(\Omega):=\big\{u\in H^s(\R^n)\::\: u=0 \text{ in }\R^n\backslash \Omega \big\} \qquad \text{ for any } s>0
\end{align*}
and~$H^s(\R^n)$ denotes the usual fractional Sobolev space of order~$s>0$ (see, for example,~\cite{proc}, for standard existence and uniqueness results in this setting).  If~$s=m\in\N$, then~$\cH^s_0(\Omega)$ is the usual Sobolev space~$H^m_0(\Omega)$.

\begin{theorem}\label{torsion:thm} 
Let~$n\geq 2$,~$s>0$,~$A\in\R^{n\times n}$ be a symmetric positive definite matrix, and let
\begin{align*}
E:=\{x\in\R^n\::\: Ax\cdot x<1 \}. 
\end{align*}
Then, there is~$\kappa=\kappa(n,s,A)>0$ such that~$u_s:\R^n\to\R$ given by~$u_s(x):=(1-Ax\cdot x)^s_+$ solves pointwisely
\begin{align}\label{t:p}
 \Ds u_s = \kappa\quad \text{ in }E,
\end{align}
and~$u_s$ is the unique (weak) solution of~\eqref{t:p} in~$\cH^s_0(E)$.
\end{theorem}

Here~$f_+$ denotes the positive part of~$f$. The explicit value of~$\kappa(n,s,a)$ can be computed in terms of hypergeometric functions~$\hf$ (see~\eqref{i},~\eqref{J0}, and~\eqref{torsion}). In particular, for (two-dimensional) ellipses with axes of length~$\frac{1}{\sqrt{a_1}}$ and~$\frac{1}{\sqrt{a_2}}$ we have that
\begin{align*}
 \kappa=4^{s}\Gamma(1+s)^2 a_1^{s+\frac{1}{2}} a_2^{-1/2} \hf\Big(s+1, \frac12; 1;1-\frac{a_1}{a_2}\Big)
    \qquad \text{ for }a_1, a_2>0,
\end{align*}
see Remark~\ref{t:rem}. The name \textit{torsion function} comes from elasticity theory, where~$u_1$ denotes the \textit{Prandtl torsion stress function} describing the deformation of an elastic body subject to surface forces. The function~$u_1$ also has applications in fluid mechanics (modelling the pressure gradient of a flow in a viscous fluid), see~\cite{km} and the references therein. A solution of~\eqref{t:p} in general domains for any~$s>0$ is usually also called torsion function, and its explicit expression is often useful for checking inequalities and to formulate or disprove general conjectures (see, for example,~\cite{km,rj,jsw}).

Theorem~\ref{torsion:thm} relies on the following more general result, which is an extension of~\cite[Corollary 4]{meijer} to ellipsoidal domains. 

\begin{theorem}\label{thm:s+j-intro}
Under the assumptions of Theorem~\ref{torsion:thm}, let~$j\in\Z$,~$j\geq-\lfloor s\rfloor-1$,
and~$u_{s+j}(x)=(1-Ax\cdot x)^{s+j}_+,\ x\in\R^n$. Then~$u_{s+j}$ solves pointwisely
\begin{align*}
\Ds u_{s+j}=f_j \quad\text{in }E, \qquad u_{s+j}=0 \quad\text{in }\R^n\setminus E,
\end{align*}
where~$f_j$ is the polynomial of degree~$(2j)_+$ given by
\begin{align}\label{lapla u_s+j intro}
f_j(x)=\left\lbrace\begin{aligned}
& C_{n,s,j}
\sum_{k=0}^{j} (-1)^k\frac{\Gamma(s+\frac{1}{2}+k)}{\Gamma(\frac{1}{2}+k)}\binom{j}{k}
\int_{\partial E}\frac{\big(u_1(x)+(Ax\cdot\theta)^2\big)^{j-k} \big(Ax\cdot\theta\big)^{2k}}{\big|\theta\big|^{n+2s}|A\theta|}\;d\theta,
& \text{if } j\geq 0, \\
& 0, & \text{if } j\leq -1,
\end{aligned}\right.
\end{align}
and, under the convention~$\Gamma(t)^{-1}=0$ for~$t\in \Z\setminus \N$,
\begin{align*}
C_{n,s,j}=\frac{2^{2s-1}\Gamma(\frac{n}{2}+s)\,\Gamma(1+s+j)}{\pi^{(n-1)/2}\,\Gamma(\frac{1}{2}+s)\,\Gamma(1+j)}.
\end{align*}
\end{theorem}
Theorem~\ref{thm:s+j-intro} can be in turn deduced as a particular case of Theorem~\ref{prop:general laplacians} (see also Corollaries~\ref{cor:torsion} and~\ref{u-s-j-calculations}).
The proof of Theorem~\ref{prop:general laplacians} relies on direct computations mainly inspired by~\cite{dyda,meijer}. 

Using this approach, we can also calculate the evaluation of~$\Ds$ of functions such as
\begin{align}\label{us}
 x_i\, u_s\qquad\text{and}\qquad  x_i^2\, u_s
 \end{align}
for~$i=1,\ldots,n$, see Lemmas~\ref{polynomial-of-degree-one} and~\ref{polynomial-of-degree-two}. With a similar strategy one may compute the fractional Laplacian~$\Ds$ of~$x_i^k\, u_s$ for any~$k\in \N$ (although the length of the expressions increases considerably with~$k$).

These formulas are of independent interest since, as mentioned earlier, there are very few examples of explicit computations regarding fractional Laplacians. However, one of our main motivations in studying these expressions is related to the problem of the \textit{positivity preserving property} (\textit{p.p.p.}, from now on) for higher-order elliptic operators, which we describe next.

We say that the operator~$\Ds$ satisfies a~\textit{p.p.p.} (in~$\Omega$) if
\begin{align}\label{ppp}
 u\geq 0 \text{ a.e. in~$\Omega$, whenever }u\in \cH^s_0(\Omega) \text{ and } (-\Delta)^s u\geq 0 \text{ pointwisely in~$\Omega$}.
\end{align}
Property~\eqref{ppp} is sometimes called \textit{weak maximum principle} and it holds for general domains if~$s\in(0,1]$. The \textit{p.p.p.} is one of the cornerstones in the analysis of linear and nonlinear second-order elliptic problems, and it is involved in results regarding existence of solutions, uniqueness, regularity, symmetry, monotonicity, geometry of level sets, \textit{etc}.

Whenever~$s>1$, the verification of~\eqref{ppp} is a delicate issue; it can be shown that~\eqref{ppp} holds for any~$s>0$ whenever~$\Omega$ is a ball~\cite{na,dg} or a halfspace~\cite{dcds}; however,~\eqref{ppp} does not hold in general.  For~$s>1$, the validity of~\eqref{ppp} depends strongly on the geometry of~$\Omega$, but hitherto there is no way of knowing which domains satisfy~\eqref{ppp} and which ones do not.  The classification of domains satisfying~\eqref{ppp} is a long-standing open problem in the theory of higher-order elliptic equations, see~\cite[Section 1.2]{ggs}.

One way of approaching this problem is to find first some examples of domains where~\eqref{ppp} does \textit{not} hold, and to try to identify a common nature. In particular, the ellipse is known to be incompatible with the~\textit{p.p.p.} whenever it is eccentric enough.  This striking example shows that convexity, smoothness, and symmetry are not properties that guarantee the validity of~\eqref{ppp}.  Next we include a list of references concerned with ellipses and the absence of a~\textit{p.p.p.}:

\begin{enumerate}[label=\it\roman*)]
\item The first available result dates back to~\cite{garabedian} for the bilaplacian~$s=2$ in dimension~$n=2$, where it is shown that an ellipse with axes ratio~$5/3$ does not satisfy~\eqref{ppp}. Later, in~\cite{HJS}, it is mentioned a ratio of about~$1.17$ is enough.
\item In~\cite{nakai-sario} a machinery is designed to extend the two-dimensional examples to higher dimensions. We remark that this approach strongly relies on a separation of variables that is not available for the fractional Laplacian~\eqref{Ds:def}.
\item  For~$s=n=2$,~\cite{shapiro-tegmark} builds an explicit and elementary example: an ellipse with axes ratio equal to~$5$; the explicit sign-changing solution is a polynomial of degree~$7$. 
\item A thorough analysis for~$s=n=2$ is performed in~\cite{rg12}, finding a counterexample in terms of a polynomial of degree~$6$ in an ellipse with axes ratio equal to~$\sqrt{19}\approx 4.359$. The authors also show that it is not possible to construct a counterexample in an ellipse with polynomials of degree less than 6; moreover, it is also shown that counterexamples with degree 6 polynomials are only possible if the axes ratio is larger than~$\approx 4.352$ (this threshold also appears in our analysis, see Section~\ref{ca:sec}).
\item The first example for~$s=3$ and~$n=2$ was given in~\cite{sweers}: in this case, the ellipse has an axes ratio equal to~$12$ and the explicit sign-changing solution is a polynomial of degree~$8$.
\item Finally,~\cite{sweers-correction} suggests that, for~$s=4$ and the same ellipse as in~\cite{sweers}, 
it is possible to find an explicit nodal solution which is a polynomial of degree~$12$.
\end{enumerate}

Other domains where a general~\textit{p.p.p.} fails are some domains with corners~\cite{cd80} (in particular squares), cones~\cite{kkm89}, domains with holes~\cite{gr13}, elongated rectangles~\cite{d49}, some large cylindrical domains~\cite{gs2}, and some lima\c{c}ons and cardioids~\cite{dw05}.  For a survey on this subject for the bilaplacian in the context of the ``Boggio-Hadamard conjecture'', we refer  to~\cite[Section 1.2]{ggs} and the references therein.

All the techniques mentioned above are either incompatible or very hard to extend to the fractional setting~$s\in(0,\infty)\backslash \N$, this case requires new ideas. Nevertheless, we believe that the study of~\textit{p.p.p.} in the fractional regime is relevant, since it offers a novel perspective on the subject using the continuity of the solution mapping, see~\cite{jsw}.

For fractional powers there is only one known counterexample to~\eqref{ppp}, given in~\cite{proc} (see also~\cite[Theorem 1.11]{na}), where it is shown that, for~$s\in (k,k+1)$ with~$k$ a positive odd integer, two disjoint balls and dumbbell shaped domains do not satisfy~\textit{p.p.p.}

In the following, we show that, using our explicit computations in ellipsoids, we can construct counterexamples to~\eqref{ppp} in any dimension~$n\geq 2$ and for~$s\in(1,\sqrt3+3/2)$, where~$\sqrt3+3/2\approx 3.232$. We follow the ideas from the above mentioned paper~\cite{shapiro-tegmark}, where a counterexample in ellipses is built in terms of an explicit polynomial.  For~$n\geq 2$, let
\begin{align}\label{Ea:def}
E_a&:=
\begin{cases}
\big\{x=(x_1,\ldots,x_n)\in\R^n\::\: \sum_{i=1}^n a_i\, x_i^2<1\big\}, &\text{ if }a\in\R^n\text{ with }a_i>0,\\
\big\{x=(x_1,\ldots,x_n)\in\R^n\::\: \sum_{i=1}^{n-1} x_i^2+ax_n^2<1\big\}, &\text{ if }a>1.
\end{cases}
\end{align}
For functions in~$\cH^s_0(E_a)\cap C^{2s+\gamma}(E_a)$ with~$\gamma>0$ and~$s>1$, the fractional Laplacian can be evaluated via the hypersingular integral~\eqref{Ds:def}, but it can also be evaluated as a composition of operators (see~\cite[Corollary 1.4]{cpaa}), namely,
\begin{align*}
 \Ds u =
 \begin{cases}
(-\Delta)(-\Delta)^{s-1} u \qquad\ \text{ for }s\in(1,2),\\
(-\Delta)^2(-\Delta)^{s-2} u \qquad \text{ for }s\in(2,3).
 \end{cases}
\end{align*}
We emphasize that the order of the differential operators \textit{cannot} be interchanged freely in the context of boundary value problems. For more details, see~\cite{cpaa,survey}.

\begin{theorem}\label{ce:intro:thm} Let~$n\geq 2$ and~$s\in (1,2)$. There are~$a_0=a_0(s,n)>1$ and~$\eps_0=\eps_0(s,n)\in(0,1)$ such that, for every~$a>a_0$ and~$\eps\in(0,\eps_0)$, the function~$U_\eps:\R^n\to\R$ given by 
\begin{align*}
U_{\epsilon}(x)&:=\big((1-x_1)^2-\epsilon\big)\bigg( 1-\sum_{i=1}^{{n-1}} x_i^2 - ax_n^2\, \bigg)^s_+,
\qquad x\in\R^n,
\end{align*}
belongs to~$\cH^s_0(E_a)$, it changes sign in~$E_a$, and~$\Ds U_{\epsilon}>0$ in~$E_a$.
\end{theorem}

For larger values of~$s$ one can still construct a counterexample, but the shape of~$U_\eps$ is slightly more involved. 
\begin{theorem}\label{thm:ce:ext} Let~$n\geq 2$ and~$s\in (1,\sqrt3+3/2)$. There are constants~$a_0>1$,~$\eps_0\in(0,1)$,~$\gamma\geq 0$, and~$\delta\geq 0$, depending only on~$s$ and~$n$, such that, for every~$a>a_0$ and~$\eps\in(0,\eps_0)$, the function~$U_\eps:\R^n\to\R$ given by 
\begin{align}\label{po}
\begin{split}
U_{\epsilon}(x) &:=\big(p(x)-\eps\big)\Big( 1-\sum_{i=1}^{{n-1}} x_i^2 - ax_n^2\, \Big)^s_+, \\
\text{where }\ p(x) &:=(1-x_1)^2+\gamma(1-x_1)-\delta \Big(\sum\limits_{k=2}^{n-1}x_k^2+ax_n^2\Big),
\end{split}
\qquad\ x\in\R^n,
\end{align}
belongs to~$\cH^s_0(E_a)$,~$U_{\epsilon}$ changes sign in~$E_a$, and~$\Ds U_{\epsilon}>0$ in~$E_a$.
\end{theorem}

We emphasize that Theorem~\ref{thm:ce:ext} is the first counterexample to~\eqref{ppp} in the range~$s\in(2,3)$.  In contrast to the results in~\cite{shapiro-tegmark} and~\cite{sweers} which rely on explicit computations of polynomials that can be verified quickly with a computer, the fractional case is much more complex, even with the explicit form of the fractional Laplacian~$(-\Delta)^sU_\eps$, since these formulas are given in terms of hypergeometric functions which are in general difficult to manipulate.  To overcome this difficulty, we use an asymptotic analysis as the length of \textit{one} of the axis in the ellipsoid goes to zero; it turns out that a suitable normalization of the hypergeometric functions simplifies in the limit and its asymptotic behavior can be determined with precision (see Lemma~\ref{asymptotic-ji}).  This is enough to guarantee the positivity of~$(-\Delta)^sU_\eps$ for thin enough ellipsoids. 

As to the upper bound~$\sqrt{3}+3/2$ for~$s$ in Theorem~\ref{thm:ce:ext}, it is a technical limitation of our asymptotic approach involving polynomials of the form~\eqref{po}. Surprisingly, for some (relatively) small values of~$a$ one can obtain counterexamples for slightly larger~$s$ (up to around 3.8), and we explore this fact in Section~\ref{ca:sec}, where we do a computer-assisted analysis in two dimensions. We also remark that, as expected,~$a_0\uparrow\infty$  as~$s\downarrow 1$, as can be seen in Figure~\ref{f2}.

We believe that counterexamples for any~$s>3$ can be found in suitable ellipses,  but this requires a more involved analysis with polynomials~$p$ of degree strictly higher than two, and we do not pursue this here. See the discussion in Section~\ref{ca:sec} and see~\cite{sweers-correction} for a counterexample to the~\textit{p.p.p.} for~$s=4$ in terms of a polynomial of degree 12.  

\medskip

Via a \textit{point inversion transformation}, one can use Theorem~\ref{thm:ce:ext} to show that a wide variety of shapes do not satisfy~\eqref{ppp} either.  To be more precise, in~\cite{dcds} (see also~\cite{dg}) the following result is shown.

\begin{proposition}[Proposition 1.6 in~\cite{dcds}]\label{lem:sharm} 
Let~$v\in \R^n$,~$c,s>0$,~$u\in C^{\infty}_c(\R^n\setminus\{-\nu\})$, and~$x\in \R^n\setminus\{-\nu\}$. Then
\begin{align}\label{kelvin-trafo}
\Ds\Big(\frac{u\circ \sigma (x)}{|x+\nu|^{n-2s}}\Big)=c^{2s}\,\frac{(-\Delta)^su(\sigma (x))}{|x+\nu|^{n+2s}},\qquad\text{  where }\quad \sigma (x):=c\,\frac{x+\nu}{|x+\nu|^{2}}-\nu.
\end{align}
\end{proposition}

To understand the geometrical meaning of the point inversion transformation~$\sigma$, see Figure~\ref{pinv}.  We have the following consequences of Theorems~\ref{prop:general laplacians},~\ref{thm:ce:ext}, and Proposition~\ref{lem:sharm}. Let~$n\geq 1$,~$c>0$,~$\nu\in \R^n\backslash \partial E_a$, and

\begin{align}\label{Omega}
 \Omega(a,c,\nu):= 
\left\lbrace\begin{aligned}
& \left\{x\in\R^n\::\: \sum_{i=1}^n a_i\left(c\frac{x_i+\nu_i}{|x+\nu|^{2}}-\nu_i\right)^2<1\right\}, &\text{ if }a\in\R^n\text{ with }a_i>0,\\
& \left\{x\in\R^n\::\: \sum_{i=1}^{{{n-1}}}\left(c\frac{x_i+\nu_i}{|x+\nu|^{2}}-\nu_i\right)^2+
a\left(c\frac{x_n+\nu_n}{|x+\nu|^{2}}-\nu_n\right)^2
<1\right\}, &\text{ if }a>1.
\end{aligned}\right.
\end{align}

\begin{corollary}\label{omega:cor}Let~$n\geq 1$,~$c>0$,~$a\in \R^n$ with~$a_i>0$, and~$\nu\in \R^n\backslash \partial E_a$. Then~$-\nu\not\in\overline{\Omega(a,c,\nu)}$ and, for~$s>0$, the function
\begin{align}\label{ws}
 w_s(x):=\frac{1}{|x+\nu|^{n-2s}}\left( 1-\sum_{i=1}^n
 a_i\left(c\frac{x_i+\nu_i}{|x+\nu|^{2}}-\nu_i\right)^2
 \, \right)^s_+,
 \qquad x\in\R^n,
\end{align}
is a pointwise solution of 
\begin{align}\label{rhs}
 (-\Delta)^s w_s(x) = \frac{k}{|x+\nu|^{n+2s}}\quad \text{ in }\Omega(a,c,\nu),\qquad w_s=0\quad \text{ in }\R^n\setminus\Omega(a,c,\nu)
\end{align}
for some constant~$k=k(n,s,c,a)>0$.
\end{corollary}

\begin{corollary}\label{omega:cor2}
 Let~$n\geq 2$,~$a,c>0$, and~$\nu\in \R^n\backslash \partial E_a$ such that~$\Omega(a,c,\nu)$ is a bounded domain. Then~$-\nu\not\in \overline{\Omega(a,c,\nu)}$ and, for every~$s\in (1,\sqrt3+3/2)$, there is~$a_0=a_0(s,n)>1$ such that~$\Omega(a,c,\nu)$ does not satisfy~\eqref{ppp} for every~$a>a_0$. For the case~$\Omega(a,c,\nu)$ unbounded, the claim still holds under the assumption~$n>4s$.
\end{corollary}

To see some of the different (bounded and unbounded) domains represented by~$\Omega(a,c,\nu)$ for~$n=2$ and~$n=3$, see Figures~\ref{fig1} and~\ref{fig2} in Section~\ref{shapes}.

\medskip

The paper is organized as follows.  In Section~\ref{sec:notation} we introduce some of the most relevant notation and important definitions. In Section~\ref{sec2} we show Theorems~\ref{torsion:thm} and~\ref{thm:s+j-intro} and deduce the explicit formulas regarding functions of the type~\eqref{us} in ellipsoids.  Section~\ref{sec:ce} is devoted to the construction of counterexamples, and contains the proofs of Theorems~\ref{ce:intro:thm} and~\ref{thm:ce:ext}, as well as those of Corollaries~\ref{omega:cor} and~\ref{omega:cor2}. 

\section{Notation and definitions}\label{sec:notation}

\subsection{The higher-order fractional Laplacian}
Any positive power~$s>0$ of the (minus) Laplacian, \textit{i.e.}~$\Ds$, has the same Fourier symbol (see~\cite[Chapter 5]{SKM93} or~\cite[Theorem 1.8]{cpaa}) as the following hypersingular integral,
\begin{align}\label{Ds:def}
L_{m,s} u(x):=\frac{c_{n,m,s}}{2}\int_{\R^n} \frac{\delta_m u(x,y)}{|y|^{n+2s}} \ dy,\qquad x\in \R^n,
\end{align}
where~$n\in\N$ is the dimension,~$m\in\N$,~$s\in(0,m)$,
\begin{align*}
\delta_m u(x,y):= \sum_{k=-m}^m (-1)^k { \binom{2m}{m-k}} u(x+ky) \qquad \text{ for }x,y\in \R^n
\end{align*}
is a finite difference of order~$2m$, and~$c_{n,m,s}$ is the positive constant given by
\begin{align}\label{cNms:def}
c_{n,m,s}
:=\left\{\begin{aligned}
&\frac{4^{s}\Gamma(\frac{n}{2}+s)}{ \pi^{n/2}\Gamma(-s) }\Big(\sum_{k=1}^{m}(-1)^{k}{ \binom{2m}{m-k}} k^{2s}\Big)^{-1},&& s\in(0,m)\backslash \N,\\
&\frac{4^{s}\Gamma(\frac{n}{2}+s)s!}{2\pi^{n/2}} \Big(\sum_{k=2}^m (-1)^{k-s+1}{\binom{2m}{m-k}} k^{2s} \ln(k)\Big)^{-1},&& s\in\{1,\ldots,m-1\}.
\end{aligned}\right.
\end{align}
In particular, if~$\lfloor s \rfloor$ denotes the floor of~$s$, then
\begin{align*}
(-\Delta)^s u(x) = (-\Delta)^{\lfloor s \rfloor}(-\Delta)^{s-\lfloor s \rfloor} u(x) = L_{m,s} u(x) 
\end{align*}
for~$x\in\Omega$ and for any~$u\in C^{2s+\beta}(\Omega)\cap \cH^s_0(\Omega)$, with~$\beta>0$, see~\cite[Corollary 1.4]{cpaa}.

\subsection{Ellipsoids}\label{not:ellipse}

Let~$n\geq 1$,~$a\in \R^n$,~$a_i>0$, and~$A=\text{diag}{(a_k)}_{k=1}^n$ a diagonal matrix.  Then, for~$x,y\in \R^n$,
\begin{align*}
 \langle x \,,\, y\rangle_a:=Ax\cdot y\qquad\text{ and }\qquad |x|_a:=\sqrt{\langle x\,,\, x\rangle_a}
\end{align*}
define an equivalent scalar product and norm in~$\R^n$ (note that the converse is also true for any symmetric positive definite matrix~$A$, after a suitable rotation of the axes). Let~$E_a\subset\R^n$ denote the open unitary ball with respect to the~$a$-norm, \textit{i.e.},
\begin{align*}
E_a:=\{x\in\R^n:|x|_a<1\}.
\end{align*}
In Section~\ref{sec:ce} we use~$a$ to denote a positive real number, in this case we use the convention given in~\eqref{Ea:def}. 

For~$\beta>-1$, let the function~$u_\beta:\R^n\to\R$ be given by
\begin{align*}
 u_\beta(x):=\big(1-|x|^2_a\big)^\beta_+, 
\qquad x\in\R^n.
\end{align*}
We also let
\begin{equation}\label{constant-measure}
\mu(d\theta)=\frac{d\theta}{{|\theta|}^{n+2s}|A\theta|},
\end{equation}
where~$d\theta$ denotes the surface measure of~$\partial E_a$, and
\begin{align}
J_0 &:= \int_{\partial E_a}\mu(d\theta), \label{J0} \\
J_i^{(k)} & := a_k^i\int_{\partial E_a}\theta_k^{2i} \;\mu(d\theta),
\qquad k\in\{1,\ldots,n\},\ i\in\N. \label{Ji}
\end{align}
These integrals appear frequently in our explicit evaluations.  In the particular case~$a_1=\ldots=a_{n-1}=1$, the integrals~$J_0$ and~$J_i^{(k)}$ can be computed explicitly as well as their asymptotic profile as~$a_n\uparrow\infty$, see Lemma~\ref{asymptotic-ji}.

\subsection{Special functions}\label{not:hf} 

We use the gamma, beta, and hypergeometric functions in our analysis, see~\cite[Chapter 6 and Chapter 15]{abramowitz} for general properties of these functions.  We collect here the definitions and some integral representations.
\begin{enumerate}
\item \textit{(Gamma function)} For~$z>0$ we denote by 
\begin{align*}
\Gamma(z)=\int_0^{\infty} t^{z-1}e^{-t}\;dt
\end{align*} 
the gamma function. If~$z\in(-\infty,0)\setminus \Z$, we let~$\Gamma(z)$ be given by the iterative definition~$\Gamma(z+1)=z\Gamma(z)$.
\item \textit{(Beta function)} For~$a,b>0$ we denote by 
\begin{align*}
B(a,b)=\frac{\Gamma(a)\Gamma(b)}{\Gamma(a+b)}
\end{align*} 
the beta function. Note that in this case
\begin{align*}
B(a,b)=\int_0^{1}(1-t)^{a-1}t^{b-1}\;dt=\int_0^{\infty}\frac{t^{a-1}}{(1+t)^{a+b}}\;dt.
\end{align*}
\item \textit{(Hypergeometric function)} For~$a,b,c,z\in\R$ with~$|z|<1$,~$\hf(a,b;c;z)$ denotes the hypergeometric function 
\begin{align}\label{s}
\hf(a,b;c;z):=\sum_{k=0}^{\infty} \frac{(a)_k(b)_k}{(c)_k}\frac{z^k}{k!},
\end{align}
where~$(q)_k$ is the Pochhammer symbol given by~$(q)_0=1$ and~$(q)_k=\prod_{i=0}^{k-1}(q+i)$. Note that if~$q\notin\Z\cap(-\infty,0]$, then~$(q)_k=\frac{\Gamma(q+k)}{\Gamma(q)}$ for~$k\in \N_0$ and hence in particular, if~$a,b,c\notin \Z\cap (-\infty,0]$, then
\begin{align*}
\hf(a,b;c;z)= \frac{\Gamma(c)}{\Gamma(a)\Gamma(b)}\sum_{k=0}^{\infty} \frac{\Gamma(a+k)\Gamma(b+k)}{\Gamma(c+k)}\frac{z^k}{k!}.
\end{align*}
If instead~$q\in\Z\cap(-\infty,0]$, then
\begin{align}\label{qk}
(q)_k=0\qquad \text{ for~$k+q\geq 1$}.
\end{align}
Moreover, if~$c>b>0$, then by using the meromorphic extension of the hypergeometric function we have for~$z<1$
\begin{align}\label{i}
\hf(a,b;c;z)=\frac{1}{B(b,c-b)}\int_0^1 t^{b-1}(1-t)^{c-b-1}(1-zt)^{-a}\;dt.
\end{align}
\end{enumerate}

\section{Explicit evaluations}\label{sec2}

\begin{lemma}\label{lem:recurrence}
Let~$s>0$ and~$\beta>0$. Then, for~$i,j,k\in\{1,\ldots,n\}$ and~$x\in E_a$,
\begin{align}
\Ds\big(x_i\, u_\beta(x)\big) & = -\frac1{2(\beta+1)a_i}\;\partial_i\Ds  u_{\beta+1}(x), \label{lapla 1u-beta} \\
\Ds\big( x_ix_j u_\beta(x) \big) & = \frac1{2(\beta+1) a_i}\bigg(\delta_{i,j}\Ds  u_{\beta+1}(x)+\frac1{2(\beta+2)a_j}\;\partial_{ij}\Ds  u_{\beta+2}(x)\bigg), \label{lapla 2u-beta} 
\end{align}
where~$\delta_{i,j}$ is the Kronecker delta. In particular,
\begin{align}
\Ds\big(x_1\, u_\beta(x)\big) & = -\frac1{2(\beta+1)a_1}\;\partial_1\Ds  u_{\beta+1}(x) \label{lapla 1u-beta bis},\\
\Ds\big( x_1^2 u_\beta(x) \big) & = \frac1{2(\beta+1) a_1}\bigg(\Ds  u_{\beta+1}(x)+\frac1{2(\beta+2)a_1}\;\partial_1^2\Ds  u_{\beta+2}(x)\bigg). \label{lapla 2u-beta bis} 
\end{align}
\end{lemma}
\begin{proof}
Let us first notice that, for any~$\beta>0$ and~$x\in E_a$,
\begin{align}
\partial_i  u_{\beta+1}(x) &= -2(\beta+1)\big(1-|x|_a^2\big)_+^{\beta} (Ax)_i=-2(\beta+1) a_i \, x_i u_\beta(x), \label{deriv 1u-beta} \\
\partial_i \big(x_j  u_{\beta+1}(x)\big) &= \delta_{i,j} u_{\beta+1}(x) -2(\beta+1) a_i \, x_ix_j u_\beta(x).\label{deriv 2u-beta} 
\end{align}
Identity~\eqref{deriv 1u-beta} directly gives~\eqref{lapla 1u-beta}.
Iterating the same idea, from~\eqref{deriv 2u-beta} one deduces
\begin{multline*}
\Ds \big( x_ix_j u_\beta(x) \big)
 = \frac1{2(\beta+1) a_i}\bigg(\delta_{i,j}\Ds  u_{\beta+1}(x)-\partial_i \Ds\big(x_j  u_{\beta+1}(x)\big)\bigg) = \\
 =\ \frac1{2(\beta+1) a_i}\bigg(\delta_{i,j}\Ds  u_{\beta+1}(x)+\frac1{2(\beta+2)a_j}\;\partial_{ij}\Ds  u_{\beta+2}(x)\bigg).
\end{multline*}
Note that the interchange between derivative~$\partial_i$ and fractional Laplacian~$(-\Delta)^s$ is allowed in this case by the Lebesgue dominated convergence theorem, see for example~\cite[Proposition B.2]{na}.
\end{proof}

\begin{theorem}\label{prop:general laplacians}
Let~$s>0$ and~$\beta>-1$.  Then
\begin{align}\label{eq:general laplacians}
\begin{split}
\Ds  u_\beta(x)= & \ \frac{2^{2s-1}\Gamma\big(\frac12+s\big)\Gamma(1+\beta)\,c_{n,m,s}}{\Gamma(1+\beta-s)\Gamma\big(\frac12\big)\,c_{1,m,s}}\ \times \\
& \times\ \int_{\partial E_a}\big(  u_1(x)+\langle x,\theta\rangle_a^2\big)^{\beta-s} \;
\hf\Big(s+\frac12,-\beta+s;\frac12;\frac{\langle x,\theta\rangle_a^2}{ u_1(x)+\langle x,\theta\rangle_a^2}\Big) \; \mu(d\theta)
\end{split}
\qquad \text{for }x\in E_a,
\end{align}
where~$c_{n,m,s}$ is given in~\eqref{cNms:def}. Here,~$\Gamma(t)^{-1}=0$ if~$t\in \Z\setminus \N$.
\end{theorem}
\begin{proof}
We consider spherical coordinates with respect to the~$a$-norm by writing
any~$z\in\R^n$ as~$z=t\theta$ with~$t>0$ and~$\theta\in\partial E_a$. This transformation has the Jacobian~$t^{n-1}/|A\theta|$, since, by the coarea formula (notice that~$\nabla|x|_a=Ax/|x|_a$),
\begin{align*}
\int_{\R^n}f(x)\;dx=\int_{\R^n}\frac{f(x)}{\big|\nabla|x|_a\big|}\big|\nabla|x|_a\big|\;dx
=\int_0^\infty\int_{t\partial E_a}\frac{f(x)\,|x|_a}{|Ax|}\;dx\;dt
=\int_0^\infty\int_{\partial E_a}\frac{f(t\theta)}{|A\theta|} \; d\theta\; t^{n-1} \;dt.
\end{align*}

We recall notation~\eqref{constant-measure} and write
\begin{align*}
\Ds  u_\beta(x) = \frac{c_{n,m,s}}2\int_{\R^n}\frac{\delta_m  u_\beta(x,y)}{{|y|}^{n+2s}}\;dy=
\frac{c_{n,m,s}}2\int_{\partial E_a}\int_0^\infty\frac{\delta_m  u_\beta(x,t\theta)}{t^{1+2s}}\;dt\;\mu(d\theta) 
= \frac{c_{n,m,s}}4\int_{\partial E_a}\int_\R\frac{\delta_m  u_\beta(x,t\theta)}{{|t|}^{1+2s}}\;dt\;\mu(d\theta).
\end{align*}
We now focus on the inner integral: recall that
\begin{align*}		
\delta_m  u_\beta(x,t\theta)=\sum_{k=-m}^m(-1)^k\binom{2m}{m-k} u_\beta(x+kt\theta)
=\sum_{k=-m}^m(-1)^k\binom{2m}{m-k}\big(1-|x+kt\theta|_a^2\big)_+^\beta.
\end{align*}
Apply the change of variables
\begin{align*}
t &= -\langle x,\theta\rangle_a + \tau \sqrt{1-|x|_a^2+\langle x,\theta\rangle_a^2},
\end{align*}
rearrange
\begin{align*}
1- & |x+kt\theta|_a^2=1-\Big|x-k\langle x,\theta\rangle_a\theta+k\tau\theta \sqrt{1-|x|_a^2+\langle x,\theta\rangle_a^2}\Big|_a^2= \\
=\ & 1-|x|_a^2-k^2\langle x,\theta\rangle_a^2-k^2\tau^2\big( 1-|x|_a^2+\langle x,\theta\rangle_a^2\big)+2k\langle x,\theta\rangle_a^2
-2k(1-k)\langle x,\theta\rangle_a^2\tau\sqrt{1-|x|_a^2+\langle x,\theta\rangle_a^2} \\
=\ & \bigg(
\frac{1-|x|_a^2}{1-|x|_a^2+\langle x,\theta\rangle_a^2}+(2k-k^2)\frac{\langle x,\theta\rangle_a^2}{1-|x|_a^2+\langle x,\theta\rangle_a^2}-k^2\tau^2
-2k(1-k)\tau\frac{\langle x,\theta\rangle_a}{\sqrt{1-|x|_a^2+\langle x,\theta\rangle_a^2}}
\bigg)\times \\
 & \times \big( 1-|x|_a^2+\langle x,\theta\rangle_a^2\big) \\
=\ & \bigg(
1-(1-k)^2\frac{\langle x,\theta\rangle_a^2}{1-|x|_a^2+\langle x,\theta\rangle_a^2}-k^2\tau^2
-2k(1-k)\tau\frac{\langle x,\theta\rangle_a}{\sqrt{1-|x|_a^2+\langle x,\theta\rangle_a^2}}
\bigg) \big( 1-|x|_a^2+\langle x,\theta\rangle_a^2\big) \\
=\ & \bigg(1-\Big((1-k)\frac{\langle x,\theta\rangle_a}{\sqrt{1-|x|_a^2+\langle x,\theta\rangle_a^2}}+k\tau\Big)^2\bigg)\big( 1-|x|_a^2+\langle x,\theta\rangle_a^2\big),
\end{align*}
and deduce
\begin{multline*}
\int_\R\frac{\delta_m u_\beta(x,t\theta)}{{|t|}^{1+2s}}\;dt=\big( 1-|x|_a^2+\langle x,\theta\rangle_a^2\big)^{\beta-s} \ \times \\
\times\ \int_\R\frac{\displaystyle\sum_{k=-m}^m(-1)^k\binom{2m}{m-k}\bigg(1-\Big((1-k)\frac{\langle x,\theta\rangle_a}{\sqrt{1-|x|_a^2+\langle x,\theta\rangle_a^2}}+k\tau\Big)^2\bigg)_+^\beta}{\displaystyle\Big|\tau-\frac{\langle x,\theta\rangle_a}{\sqrt{1-|x|_a^2+\langle x,\theta\rangle_a^2}}\Big|^{1+2s}}\;d\tau
\end{multline*}
which amounts to (after a translation in the~$\tau$ variable)
\begin{align}
\int_\R\frac{\delta_m u_\beta(x,t\theta)}{{|t|}^{1+2s}}\;dt = \
\big( 1-|x|_a^2+\langle x,\theta\rangle_a^2\big)^{\beta-s}\int_\R\frac{\displaystyle\sum_{k=-m}^m(-1)^k\binom{2m}{m-k}\big(1-(\widetilde x_\theta+k\tau)^2\big)_+^\beta}{{|\tau|}^{1+2s}} \; d\tau,
\label{ec1}
\end{align}
where~$\tilde x_\theta := \langle x,\theta\rangle_a \, (1-|x|_a^2+\langle x,\theta\rangle_a^2)^{-1/2}.$  Now, using a particular case of\footnote{In
the notations of~\cite[Corollary 4]{meijer}, we fix~$V(x)\equiv 1,\ l=0,\ \delta=n=1,\ \sigma=\beta$ and~$\rho=s$.}~\cite[Corollary 4]{meijer}, we know that
\begin{align}
\Ds \big(1-z^2\big)_+^{\beta}& = \frac{c_{1,m,s}}2\int_\R\frac{\displaystyle\sum_{k=-m}^m(-1)^k\binom{2m}{m-k}\big(1-(z+k\tau)^2\big)_+^\beta}{{|\tau|}^{1+2s}} \; d\tau\notag
\\
&=\frac{2^{2s}\Gamma\big(\frac12+s\big)\Gamma(1+\beta)}{\Gamma(1+\beta-s)\Gamma\big(\frac12\big)}\; \hf\Big(s+\frac12,-\beta+s;\frac12;z^2\Big)
\qquad\text{for }z\in(-1,1).\label{ec2}
\end{align}

Therefore, by~\eqref{ec1} and~\eqref{ec2},
\begin{align*}
\Ds  u_\beta(x) &= \frac{c_{n,m,s}}{4}\int_{\partial E_a}\big( 1-|x|_a^2+\langle x,\theta\rangle_a^2\big)^{\beta-s}
\int_\R\frac{\displaystyle\sum_{k=-m}^m(-1)^k\binom{2m}{m-k}\big(1-(\widetilde x_\theta+k\tau)^2\big)_+^\beta}{{|\tau|}^{1+2s}} \; d\tau \; \mu(d\theta) \\
&= \frac{c_{n,m,s}}{2c_{1,m,s}}\int_{\partial E_a}\big( 1-|x|_a^2+\langle x,\theta\rangle_a^2\big)^{\beta-s} \;
\Ds (1-\tilde x_\theta^2)^\beta_+ \; \mu(d\theta),\\&=
\frac{2^{2s-1}\Gamma\big(\frac12+s\big)\Gamma(1+\beta)\,c_{n,m,s}}{\Gamma(1+\beta-s)\Gamma\big(\frac12\big)\,c_{1,m,s}}
\int_{\partial E_a}\big( 1-|x|_a^2+\langle x,\theta\rangle_a^2\big)^{\beta-s} \;
\hf\Big(s+\frac12,-\beta+s;\frac12;\tilde x_\theta^2\Big) \; \mu(d\theta).
\end{align*}
\end{proof}

In the next corollaries we collect some consequences of Theorem~\ref{prop:general laplacians}. For this let
\begin{align}\label{kns}
k_{n,s}:=\frac{2^{2s-1}\Gamma(n/2+s) }{ \pi^{n/2} }.
\end{align}

\begin{corollary}\label{cor:torsion}
Let~$s>0$.
Then it holds
\begin{align}\label{torsion}
\Ds u_{s}(x)=\Gamma(1+s)k_{n,s}\;J_0
\qquad\text{for }x\in E_a.
\end{align}
Moreover, for any~$\ell\in\N$ such that~$s-\ell>-1$, it also holds
\begin{align}
\Ds u_{s-\ell}(x)=0
\qquad\text{for }x\in E_a,
\end{align}
with~$J_0$ as in~\eqref{J0}.
\end{corollary}
\begin{proof}
Both statements follow by just considering respectively~$\beta=s$ and~$\beta=s-\ell$ in~\eqref{eq:general laplacians}.
Note that for~\eqref{torsion} we are using that~$\hf\Big(s+\frac12,0;\frac12; t \Big)=1$ for~$t\in (-1,1)$ and, moreover, since 
\begin{align*}
c_{n,m,s} = \frac{4^s \Gamma(n/2+s)}{\pi^{n/2}\Gamma(-s)}\bigg(\sum_{k=-m}^m(-1)^k\binom{2m}{m-k}\bigg)^{-1}, 
\qquad s\in (0,m)\setminus\N,
\end{align*}
we have
\begin{align*}
\frac{c_{n,m,s}}{ c_{1,m,s}}=\frac{\Gamma(n/2+s)\sqrt{\pi}}{ \pi^{n/2}\Gamma(\frac{1}{2}+s)}.
\end{align*}
Note that the same holds for~$s\in \N$ and hence
\[
\frac{2^{2s-1}\Gamma\big(\frac12+s\big)\,c_{n,m,s}}{\Gamma\big(\frac12\big)\,c_{1,m,s}}= \frac{2^{2s-1}\Gamma(n/2+s) }{ \pi^{n/2} }=k_{n,s}.
\]
\end{proof}

\begin{proof}[Proof of Theorem~\ref{torsion:thm}] 
Using the rotation invariance of the fractional Laplacian, we may assume that~$A$ is a diagonal matrix.  By~\eqref{torsion}, we have that
\begin{align}\label{tau}
\tau(x):=\frac{1}{\Gamma(1+s)k_{n,s}\;J_0}\big(1-|x|_a^2\big)_+^s,
\qquad x\in\R^n,
\end{align}
satisfies pointwisely that
\begin{align}\label{tauP}
\Ds \tau(x)=1 
\quad\text{for }x\in E_a.
\end{align}
Moreover,~$\tau\in \cH^s_0(E_a)$. For~$s\in\N$ this is clear, so let~$s\notin\N$ and~$m\in \N$ such that~$s\in(m,m+1)$. We argue with the Gagliardo-Nirenberg interpolation inequality (see, \textit{e.g.},~\cite[Theorem 1]{BM18}),
\begin{equation}\label{eq:interpolation}
\|f\|_{W^{s,p}(E_a)}\leq C\|f\|_{W^{s_1,p_1}(E_a)}^{\theta}\|f\|_{W^{s_2,p_2}(E_a)}^{\theta}\qquad\text{for all~$f\in W^{s_1,p_1}(E_a)\cap W^{s_2,p_2}(E_a)$,}
\end{equation}
which for some~$C$ independent of~$f$ is satisfied for~$1< p,p_1,p_2\leq\infty$,~$0<s_1<s<s_2$ satisfying for some~$\theta\in(0,1)$ the relation
\[
s=\theta s_1+(1-\theta)s_2\quad\text{and}\quad \frac{1}{p}=\frac{\theta}{p_1}+\frac{1-\theta}{p_2}.
\]
Next note that for any~$\beta,\beta'\in\N_0^n$ with~$|\beta|=m$ and~$|\beta'|=m+1$ there is a constant~$\tilde{C}>0$ such that 
\[
|\partial^{\beta}\tau(x)|\leq \tilde{C}(1-|x|_a)^{s-m}\quad\text{and}\quad|\partial^{\beta'}\tau(x)|\leq \tilde{C}(1-|x|_a)^{s-m-1}\quad\text{for~$x\in E_a$}
\]
so that~$\tau\in W^{m,\infty}(E_a)$ and also~$\tau\in W^{m+1,p_2}(E_a)$ for~$1<p_2<\frac{1}{1+m-s}$. By~\eqref{eq:interpolation} with~$\theta=1+m-s$,~$s_1=m$,~$s_2=m+1$, and~$p_1=\infty$, we then have~$\tau\in W^{s,p}(E_a)$ for all~$p=\frac{p_2}{1-\theta}<\frac{1}{(1+m-s)(s-m)}$. Since~$(1+m-s)(s-m)\leq \frac{1}{4}$, we have in particular~$\tau \in H^s(E_a)=W^{s,2}(E_a)$. Since also~$\tau/\dist(\cdot,\partial E_a)^{s}\in L^{\infty}(E_a)$, it follows that~$\tau\in\cH^s_0(E_a)$ also for~$s\notin\N$ (see, for example~\cite[Section 4.3.2, equation 7]{T78}). But then, by uniqueness of weak solutions,~$\tau$ is the unique weak solution of~\eqref{tauP} in~$\cH^s_0(E_a)$.  
\end{proof}

\begin{remark}[Torsion function in an ellipse]\label{t:rem} In two dimensions, the constant of the torsion function~$\tau$ can be computed explicitly with a direct computation. Let~$\alpha_1,\alpha_2>0$ and~$\cE=\{x\in\R^2\::\: \alpha_1x_1^2+\alpha_2x_2^2<1\}$.  For~$a=\frac{\alpha_2}{\alpha_1}$ let~$\tau$ be given by~\eqref{tau}.  Finally, let~$\widetilde \tau:\cE\to \R$ be given by
\begin{align*}
 \widetilde\tau(x)
 :&=\alpha_1^{-s}\tau(\alpha_1^{1/2} x)
 =\frac{1}{\alpha_1^{s}\Gamma(1+s)k_{n,s}\;J_0}\big(1-\alpha_1 x_1^2 - a \alpha_1 x_2^2\big)_+^s \\
 &=\frac{1}{2^{2s}\Gamma(1+s)^2\alpha_1^{s+1/2} \alpha_2^{-1/2} \hf\Big(s+1, \frac12; 1;1-\frac{\alpha_1}{\alpha_2}\Big)}
 \big(1-\alpha_1 x_1^2 - \alpha_2 x_2^2\big)_+^s,
\end{align*}
since, by Lemma~\ref{asymptotic-ji},
\begin{align*}
 J_0=a^{-1/2}B\Big(\frac12,\frac{1}2\Big) \hf\Big(s+1, \frac12; 1;1-\frac{1}{a}\Big)
 =\left(\frac{\alpha_2}{\alpha_1}\right)^{-1/2}\pi \hf\Big(s+1, \frac12; 1;1-\frac{\alpha_1}{\alpha_2}\Big)
 .
\end{align*}
Then, for~$x\in\cE$,~$(-\Delta)^s \widetilde \tau(x)
 =(-\Delta)^s \tau(ax)=1.$
\end{remark}

The case~$\beta=s+j$ with~$j\in \N$ in Theorem~\ref{prop:general laplacians} is particularly useful, and therefore we state it as a corollary.

\begin{corollary}\label{u-s-j-calculations}
Let~$j\in\N$. Then, for~$x\in E_a$,
\begin{align}\label{lapla u_s+j}
\Ds  u_{s+j}(x)=  \ \frac{\Gamma(1+s+j)}{\Gamma(1+j)}k_{n,s}\int_{\partial E_a}\big(  u_1(x)+\langle x,\theta\rangle_a^2\big)^{j} \;
\hf\Big(s+\frac12,-j;\frac12;\frac{\langle x,\theta\rangle_a^2}{ u_1(x)+\langle x,\theta\rangle_a^2}\Big) \; \mu(d\theta).
\end{align}
In the particular cases~$j=1,2$, Table~\ref{table:lapla u_s+j} hold.
\begin{table}[ht]
\centering
\begin{tabular}{c||c}
$\ j\ $ &~$\Gamma(1+s+j)^{-1}k_{n,s}^{-1}\Ds  u_{s+j}(x)$ \ for \ $x\in E_a$ \\ [.25em]
\hline \\ [-.5em]
$-\lfloor s\rfloor-1$ &~$0$ \\ 
$\vdots$ &~$\vdots$ \\ [.5em]
$-1$ &~$0$ \\ [1em]
$0$ &~$\displaystyle\int_{\partial E_a}\mu(d\theta)$ \\ [2em]
$1$ &~$\displaystyle
\int_{\partial E_a}\Big( u_1(x)-2s\langle x,\theta\rangle^2_a \Big)\mu(d\theta)$ \\ [2em]
$2$ &~$\displaystyle\frac12
\int_{\partial E_a}\Big( u_1(x)^2-4s u_1(x)\langle x,\theta\rangle_a^2+\frac{4s(s-1)}{3}\langle x,\theta\rangle_a^4\Big)\mu(d\theta)$ 
\end{tabular}
\caption{Significant examples.}\label{table:lapla u_s+j}
\end{table}
\end{corollary}
\begin{proof} 
Identity~\eqref{lapla u_s+j} simply follows by considering~$\beta=s+j$ in~\eqref{eq:general laplacians}.
In order to deduce the particular cases listed in Table~\ref{table:lapla u_s+j},
we need to remark that, as one of the arguments in the hypergeometric function is a negative integer,
then the hypergeometric function reduces to a polynomial, see~\eqref{qk}. 
Such polynomials for~$j=1,2$ can be found in Table~\ref{table:hypergeometric}.
The calculation of~$k_{n,s}$ follows as in the proof of Corollary~\ref{cor:torsion}.
\begin{table}[ht]
\centering
\begin{tabular}{c||c}
$\ j\ $ &~$\displaystyle \big(v+w \big)^{j}\hf\Big(s+\frac12,-j;\frac12;\frac{w}{v+w}\Big)$ \ for \ $t\in(-1,1)$ \\ [1em]
\hline \\ [-.5em]
$1$ &~$v-2sw$  \\ [1em]
$2$ &~$\displaystyle v^2-4svw+\frac{4s(s-1)}3w^2$
\end{tabular}
\caption{The explicit polynomial form of the hypergeometric expression.}\label{table:hypergeometric}
\end{table}
\end{proof}

\begin{proof}[Proof of Theorem~\ref{thm:s+j-intro}] 
Using the rotation invariance of the fractional Laplacian, we may assume that~$A$ is a diagonal matrix.  As mentioned above, this Theorem follows immediately from Corollary~\ref{cor:torsion} for~$j\in \Z\setminus \N_0$ and from Corollary~\ref{u-s-j-calculations}, since for~$j\in \N_0$ and~$v,w\geq0$ we have
\begin{multline*}
(v+w)^j \hf(s+\frac{1}{2},-j;\frac{1}{2};\frac{w}{v+w})=(v+w)^j\sum_{k=0}^{j}\frac{(s+\frac{1}{2})_k(-j)_k}{(\frac{1}{2})_kk!}\frac{w^k}{(v+w)^k} =\\
=\sum_{k=0}^{j}\frac{\Gamma(s+\frac{1}{2}+k)\Gamma(\frac{1}{2})}{\Gamma(s+\frac{1}{2})\Gamma(\frac{1}{2}+k)} \binom{j}{k}(-1)^kw^k (v+w)^{j-k}.
\end{multline*}
\end{proof}

\subsection{Auxiliary calculations for the counterexample}
Recall the definition of~$J_0$,~$J_i^{(k)}$, and~$\mu$ given in~\eqref{J0},~\eqref{Ji}, and~\eqref{constant-measure}.

\begin{lemma}\label{polynomial-of-degree-one}
Let~$U(x):={\big(1-a_k^{1/2}x_k\big)}\,u_{s}(x)$ for~$x\in\R^n,\ k\in\N.$ Then, for any~$x\in E_a$,
\begin{align}\label{polynomial-of-degree-two-part1}
\frac{\Ds U(x)}{k_{n,s}\Gamma(1+s)} = J_0-\big(J_0+2sJ_1^{(k)}\big)a_k^{1/2}x_k.
\end{align}
\end{lemma}
\begin{proof}
From Lemma~\ref{lem:recurrence} and Corollary~\ref{u-s-j-calculations} it follows that
\begin{multline*}
\frac{\Ds U(x)}{k_{n,s}\Gamma(1+s)} =J_0+\frac{a_k^{-1/2}}{2}\partial_k\int_{\partial E_a}\big( u_1(x)-2s\langle x,\theta\rangle_a^2\big)\;\mu(d\theta) \ = \\
= J_0 - a_k^{1/2} J_0 x_k - 2sa_k^{1/2}\int_{\partial E_{a}}\langle x,\theta\rangle_a\theta_k\;\mu(d\theta)=J_0-\big(J_0+2sJ_1^{(k)}\big)a_k^{1/2}x_k,
\end{multline*}
since, by symmetry,~$\int_{\partial E_{a}}\theta_j\theta_k \; \mu(d\theta)=0$ for~$j\in\{1,\ldots,n\}\setminus\{k\}.$
\end{proof}

\begin{lemma}\label{polynomial-of-degree-two}
Let~$U(x):=\big(1-a_k^{1/2}x_k\big)^2\,u_{s}(x)$ for~$x\in\R^n.$ Then for any~$x\in E_a$ we have
\begin{multline*}
\frac{\Ds U(x)}{k_{n,s}\Gamma(1+s)} =\big[J_0+5sJ_1^{(k)}+2s(s-1)J_2^{(k)}\big]a_kx_k^2 -2\big[J_0+2sJ_1^{(k)}\big]a_k^{1/2}x_k +J_0-sJ_1^{(k)} +\\
+s\sum_{\overset{i=1}{i\neq k}}^n \Big[J_1^{(k)}+2(s-1)a_ia_k\int_{\partial E_a}\theta_i^2\theta_k^2\;\mu(d\theta)\Big]a_ix_i^2.
\end{multline*}
\end{lemma}
\begin{proof}
Using~\eqref{lapla 1u-beta bis} and~\eqref{lapla 2u-beta bis} of Lemma~\ref{lem:recurrence}, we have
\begin{align}\label{ds ansatz}
\begin{split}
\Ds U(x) = \ &  
\Ds u_{s}(x)+\frac{a_k^{-1/2}}{s+1}	\, \partial_k\Ds u_{s+1}(x)+\frac{1}{2(s+1)}	\, \Ds u_{s+1}(x)\; + \\
& +\frac{1}{4(s+1)(s+2)a_k}	\, \partial_{k}^2\Ds u_{s+2}(x),
\end{split}
\qquad\ x\in E_a.
\end{align}
Using the identities in Table~\ref{table:lapla u_s+j}, we have
\begin{align}\label{first step}
\begin{split}
&\frac{\Ds U(x)}{\Gamma(1+s)\,k_{n,s}} = J_0 +\frac1{a_k^{1/2}} \, \partial_k\int_{\partial E_a}\Big( u_1(x)-2s\langle x,\theta\rangle^2_a\Big)\mu(d\theta) \ + \\
& \quad +\frac12\int_{\partial E_a}\Big( u_1(x)-2s\langle x,\theta\rangle^2_a\Big)\mu(d\theta) +
\frac1{8a_k}\partial_k^2\int_{\partial E_a}\Big( u_1(x)^2-4s u_1(x)\langle x,\theta\rangle_a^2+\frac{4s(s-1)}{3}\langle x,\theta\rangle_a^4\Big)\mu(d\theta).
\end{split}
\end{align}
In order to compute~\eqref{first step}, we consider the following differential identities
\begin{align}
& \partial_k  u_1(x)=-2a_kx_k, \qquad\qquad \ \  \partial_k^2  u_1(x)=-2a_k, 		\label{some derivatives begin} \\
& \partial_k\langle x,\theta\rangle_a^2=2a_k\langle x,\theta\rangle_a\theta_k, \qquad \partial_k^2\langle x,\theta\rangle_a^2=2a_k^2\theta_k^2, 	\notag\\
& \partial_k^2  u_1(x)^2=\partial_k\big(2 u_1(x)\,\partial_k u_1(x)\big)=2\big(\partial_k  u_1(x)\big)^2+2 u_1(x)\,\partial_k^2 u_1(x)=8a_k^2x_k^2-4a_k u_1(x), \notag\\
\begin{split}
	& \partial_k^2\big(  u_1(x)\,\langle x,\theta\rangle^2_a \big)=\langle x,\theta\rangle^2_a\,\partial_k^2 u_1(x)+2\partial_k  u_1(x) \,\partial_k\langle x,\theta\rangle^2_a+ u_1(x)\,\partial_k^2\langle x,\theta\rangle^2_a \notag\\
	&\hspace{2.4cm} =-2a_k\langle x,\theta\rangle^2_a-8a_k^2\langle x,\theta\rangle_a \theta_k x_k+2a_k^2 u_1(x)\theta_k^2,
\end{split} \\
& \partial_k^2\langle x,\theta\rangle^4_a=\partial_k\big(2\langle x,\theta\rangle^2_a\,\partial_k\langle x,\theta\rangle^2_a\big)
=2\big(\partial_k\langle x,\theta\rangle_a^2\big)^2+2\langle x,\theta\rangle^2_a\,\partial_k^2\langle x,\theta\rangle^2_a
=12a_k^2\langle x,\theta\rangle^2_a\theta_k^2. \label{some derivatives end}
\end{align}
In view of~\eqref{some derivatives begin}-\eqref{some derivatives end}, equation~\eqref{first step} can be rewritten
\begin{align*}
& \frac{\Ds U(x)}{\Gamma(1+s)\,k_{n,s}} \,  \ = \\  
& =\ J_0 -2a_k^{1/2}J_0x_k -4sa_k^{1/2}J_1^{(k)}x_k+\frac12J_0  u_1(x) -s\sum_{i=1}^n a_iJ_1^{(i)}x_i^2+\frac1{8a_k}\Big(8a_k^2J_0x_k^2-4J_0a_k u_1(x)+8sa_k\sum_{i=1}^n a_iJ_1^{(i)}x_i^2  \\
& \qquad +32sa_k^2J_1^{(k)} x_k^2-8sa_kJ_1^{(k)} u_1(x)
+16s(s-1)a_k^2\sum_{i=1}^n a_i^2x_i^2\int_{\partial E_a}\theta_i^2\theta_k^2\;\mu(d\theta) \Big)  \\
& =\ J_0 -2a_k^{1/2}J_0x_k -4sa_k^{1/2}J_1^{(k)}x_k +a_kJ_0x_k^2  \\
& \qquad +5sa_kJ_1^{(k)} x_k^2-sJ_1^{(k)}+sJ_1^{(k)}\sum_{\overset{i=1}{i\neq k}}^n a_ix_i^2
+2s(s-1)a_kx_k^2J_2^{(k)}+2s(s-1)a_k\sum_{\overset{i=1}{i\neq k}}^n a_i^2x_i^2\int_{\partial E_a}\theta_i^2\theta_k^2\;\mu(d\theta) \\
& =\ \big[J_0+5sJ_1^{(k)}+2s(s-1)J_2^{(k)}\big]a_kx_k^2 -2\big[J_0+2sJ_1^{(k)}\big]a_k^{1/2}x_k +J_0-sJ_1^{(k)} \\
& \qquad +sJ_1^{(k)}\sum_{\overset{i=1}{i\neq k}}^n a_ix_i^2+2s(s-1)a_k\sum_{\overset{i=1}{i\neq k}}^n a_i^2x_i^2\int_{\partial E_a}\theta_i^2\theta_k^2\;\mu(d\theta).
\end{align*}
\end{proof}

\begin{lemma}\label{polynomial-of-degree-twob}
Let~$U(x)=u_{s}(x)\sum_{\overset{i=1}{i\neq k}}^na_ix_i^2$ for~$x\in\R^n.$ Then, for any~$x\in E_a$,
\begin{multline*}
\frac{\Ds U(x)}{k_{n,s}\Gamma(1+s)} = 
s\Big[J_0-J_1^{(k)}+2(s-1)\big(J_1^{(k)}-J_2^{(k)}\big)\Big]a_k x_k^2-s\big(J_0-J_1^{(k)}\big) \, + \\
+\sum_{\overset{i=1}{i\neq k}}^n\Big[(s+1)J_0+4sJ_1^{(i)}-sJ_1^{(k)}+2s(s-1)J_2^{(i)}
+2s(s-1)\sum_{\overset{h=1}{h\neq k,i}}^n a_ia_h\int_{\partial E_a}\theta_h^2\theta_i^2\mu(d\theta)\Big]a_ix_i^2.
\end{multline*}
\end{lemma}
\begin{proof}
Using Lemma~\ref{lem:recurrence},
\begin{align*}
\Ds U(x)=\sum_{\overset{i=1}{i\neq k}}^n a_i\Ds\big(x_i^2u_{s}(x)\big)=\frac{1}{2(s+1)}\Big((n-1)\Ds u_{s+1}(x)+\frac{1}{2(s+2)}\sum_{\overset{i=1}{i\neq k}}^n\frac1{a_i}\partial_{i}^2\Ds u_{s+2}(x)\Big).
\end{align*}
By Table~\ref{table:lapla u_s+j} and by suitably adjusting~\eqref{some derivatives begin}-\eqref{some derivatives end} to the current situation, we deduce 
\begin{align*}
& \frac{\Ds U(x)}{\Gamma(1+s)k_{n,s}} \ = \\ 
& =\frac{n-1}2\int_{\partial E_a}\Big( u_1(x)-2s\langle x,\theta\rangle_a^2\Big)\;\mu(d\theta) 
+\sum_{\overset{i=1}{i\neq k}}^n\frac1{8a_i}\partial_i^2\int_{\partial E_a}\Big( u_1(x)^2-4s u_1(x)\langle x,\theta\rangle_a^2+\frac{4s(s-1)}{3}\langle x,\theta\rangle_a^4\Big)\mu(d\theta) \\
& =\frac{n-1}2J_0 u_1(x)-s(n-1)\sum_{j=1}^na_jJ_1^{(j)}x_j^2
+\sum_{\overset{i=1}{i\neq k}}^n\frac1{8a_i}\Big(8a_i^2J_0x_i^2-4J_0a_i u_1(x)+8sa_i\sum_{h=1}^n a_hJ_1^{(h)}x_h^2  \\
& \qquad +32sa_i^2J_1^{(i)} x_i^2-8sa_iJ_1^{(i)} u_1(x)
+16s(s-1)a_i^2\sum_{h=1}^n a_h^2x_h^2\int_{\partial E_a}\theta_h^2\theta_i^2\mu(d\theta) \Big)  \\
& =\sum_{\overset{i=1}{i\neq k}}^n\Big(a_iJ_0x_i^2 +4sa_iJ_1^{(i)} x_i^2-sJ_1^{(i)} u_1(x)
+2s(s-1)a_i\sum_{h=1}^n a_h^2x_h^2\int_{\partial E_a}\theta_h^2\theta_i^2\mu(d\theta) \Big).
\end{align*}
Observe that
 \begin{align*}
 \sum_{\overset{i=1}{i\neq k}}^n J_1^{(i)} & = \int_{\partial E_a} \big(1-a_k\theta_k^2\big)\;\mu(d\theta)=J_0-J_1^{(k)} \\
 \sum_{\overset{i=1}{i\neq k}}^na_i a_k\int_{\partial E_a}\theta_k^2\theta_i^2\mu(d\theta) & =
 \int_{\partial E_a}a_k\theta_k^2\big(1-a_k\theta_k^2\big)\mu(d\theta)=J_1^{(k)}-J_2^{(k)},
 \end{align*}
then
\begin{align*}
& \frac{\Ds U(x)}{\Gamma(1+s)k_{n,s}} \ = \\ 
& =\sum_{\overset{i=1}{i\neq k}}^n\big[J_0+4sJ_1^{(i)}+2s(s-1)J_2^{(i)}\big]a_ix_i^2-s\big(J_0-J_1^{(k)}\big) u_1(x)
+2s(s-1)\sum_{\overset{i=1}{i\neq k}}^na_i\sum_{\overset{h=1}{h\neq i}}^n a_h^2x_h^2\int_{\partial E_a}\theta_h^2\theta_i^2\mu(d\theta)  \\
& =s\Big[J_0-J_1^{(k)}+2(s-1)\big(J_1^{(k)}-J_2^{(k)}\big)\Big]a_k x_k^2-s\big(J_0-J_1^{(k)}\big) + \\
& \qquad +\sum_{\overset{i=1}{i\neq k}}^n\Big[(s+1)J_0+4sJ_1^{(i)}-sJ_1^{(k)}+2s(s-1)J_2^{(i)}
+2s(s-1)\sum_{\overset{h=1}{h\neq k,i}}^n a_ia_h\int_{\partial E_a}\theta_h^2\theta_i^2\mu(d\theta)\Big]a_ix_i^2.
\end{align*}
\end{proof}

\begin{remark}\label{rmk:case of interest}
Consider~$a_1=\ldots=a_{n-1}=1$ and~$a_n=a.$ In this particular case one has
\begin{align*}
& \sum_{i=2}^n \Big[a_i\int_{\partial E_a}\theta_i^2\theta_1^2\;\mu(d\theta)\Big]a_ix_i^2
=\sum_{i=2}^{n-1} \Big[\int_{\partial E_a}\theta_i^2\theta_1^2\;\mu(d\theta)\Big]x_i^2
+\Big[a\int_{\partial E_a}\theta_n^2\theta_1^2\;\mu(d\theta)\Big]ax_n^2 \\
& =\ \frac1{n-2}\Big[\sum_{i=2}^{n-1}\int_{\partial E_a}\theta_i^2\theta_1^2\;\mu(d\theta)\Big]\sum_{i=2}^{n-1} x_i^2
+\Big[a\int_{\partial E_a}\theta_n^2\theta_1^2\;\mu(d\theta)\Big]ax_n^2 \\
& =\ \frac1{n-2}\Big[J_1^{(1)}-J_2^{(1)}-a\int_{\partial E_a}\theta_n^2\theta_1^2\;\mu(d\theta)\Big]\sum_{i=2}^{n-1} x_i^2
+\Big[a\int_{\partial E_a}\theta_n^2\theta_1^2\;\mu(d\theta)\Big]ax_n^2 
\end{align*}
and therefore
\begin{align*}
\frac{\Ds\Big(\big(1-x_1\big)u_{s}(x)\Big)}{k_{n,s}\,\Gamma(1+s)} & = J_0-\big[J_0+2sJ_1^{(1)}\big]x_1, \\
\frac{\Ds\Big(\big(1-x_1\big)^2u_{s}(x)\Big)}{k_{n,s}\,\Gamma(1+s)} & = 
\big[J_0+5sJ_1^{(1)}+2s(s-1)J_2^{(1)}\big]x_1^2 -2\big[J_0+2sJ_1^{(1)}\big]x_1 +J_0-sJ_1^{(1)} \\
& \quad +s\sum_{i=2}^{n-1}\Big[J_1^{(1)}+2(s-1)\int_{\partial E_a}\theta_i^2\theta_1^2\;\mu(d\theta)\Big] x_i^2 \\
& \quad +s\Big[J_1^{(1)}+2(s-1)a\int_{\partial E_a}\theta_n^2\theta_1^2\;\mu(d\theta)\Big]ax_n^2, \\
\dfrac{\Ds\Big(u_{s}(x)\big(\sum\limits_{i=2}^{n-1} x_i^2+ax_n^2\big)\Big)}{k_{n,s}\,\Gamma(1+s)} & =
s\Big[J_0-J_1^{(1)}+2(s-1)\big(J_1^{(1)}-J_2^{(1)}\big)\Big]x_1^2-s\big(J_0-J_1^{(1)}\big) + \\
& \quad +\sum_{i=2}^{n-1}\Big[(s+1)J_0+4sJ_1^{(i)}-sJ_1^{(1)}+2s(s-1)J_2^{(i)} \\
& \qquad\qquad +2s(s-1)\sum_{\overset{h=2}{h\neq i}}^{n-1} \int_{\partial E_a}\theta_h^2\theta_i^2\mu(d\theta)
+2s(s-1)a\int_{\partial E_a}\theta_n^2\theta_1^2\mu(d\theta)\Big]x_i^2 \\
& \quad +\Big[(s+1)J_0+4sJ_1^{(n)}-sJ_1^{(1)} +2s(s-1)J_1^{(n)} -2s(s-1) a\int_{\partial E_a}\theta_n^2\theta_1^2\mu(d\theta)\Big]ax_n^2.
\end{align*}
Note also that, in this case,
\begin{align*}
\frac{\Ds\Big(\big(1-a^{1/2}x_n\big)u_{s}(x)\Big)}{k_{n,s}\,\Gamma(1+s)} & = J_0-\big(J_0+2sJ_1^{(n)}\big)a^{1/2}x_n \\
\frac{\Ds\Big(\big(1-a^{1/2}x_n\big)^2u_{s}(x)\Big)}{k_{n,s}\,\Gamma(1+s)} & = 
\big[J_0+5sJ_1^{(n)}+2s(s-1)J_2^{(n)}\big]ax_n^2 -2\big[J_0+2sJ_1^{(n)}\big]a^{1/2}x_n +J_0-sJ_1^{(n)} \\
& \quad +s\Big[J_1^{(n)}+\frac{2(s-1)}{n-1}\big(J_1^{(n)}-J_2^{(n)}\big)\Big]\sum_{i=1}^{n-1} x_i^2.
\end{align*}
For the sake of clarity we summarize the above in the following table for the particular case~$n=2$.
\begin{table}[ht]
\centering
\begin{tabular}{c||l}
$\ p(x)\ $ &~$\displaystyle \Gamma(1+s)^{-1}k_{n,s}^{-1}\Ds(pu_{s})(x)\ $ \ for \ $x\in E_a$ \\ [1em]
\hline\hline \\ [-.5em]
$1-x_1$ &~$J_0-\big[J_0+2sJ_1\big]x_1$ \\ [1em] \hline \\ [-.5em]
$(1-x_1)^2$ &~$\big[J_0+5sJ_1+2s(s-1)J_2\big]x_1^2 -2\big[J_0+2sJ_1\big]x_1 +J_0-sJ_1$ \\ [.5em]
&~$\displaystyle\quad +s\big[J_1+2(s-1)\big(J_1-J_2\big)\big]ax_2^2$ \\ [1em] \hline \\ [-.5em]
$ax_2^2$ &~$\displaystyle s\Big[J_0-J_1+2(s-1)\big(J_1-J_2\big)\Big]x_1^2-s\big(J_0-J_1\big)$  \\  [.5em] 
&~$\displaystyle \quad +\Big[(2s+1)(s+1)J_0-s(4s+1)J_1 +2s(s-1)J_2\Big]ax_2^2~$ \\ [1em] 
\end{tabular}
\caption{The particular case~$n=2,\ a_1=1,\ a_2=a$, where we simply write~$J_i$ for~$J_i^{(1)}$.}\label{table:some more examples}
\end{table}

\end{remark}

\section{Counterexample to positivity preserving properties in ellipsoids}\label{sec:ce}

In the following, we give a counterexample to the positivity preserving property (see~\eqref{ppp}) of~$\Ds$,~$s>1$, in an ellipsoid~$E_a$, where we choose~$a_1=\ldots=a_{n-1}=1$, and~$a_n=a>1$ sufficiently large. To this end, we consider 
\begin{align}\label{U}
U(x):=p(x)u_{s}(x),
\qquad x\in\R^n,
\end{align}
where~$p$ is a polynomial of degree two such that~$p-\epsilon$ is sign-changing for every~$\epsilon>0$. Note that once we have shown that there is a constant~$k>0$ such that
\begin{equation}\label{goal-counter}
\Ds U\geq k\qquad\text{in~$E_a$,}
\end{equation}
it follows, by linearity, that for a suitable~$\epsilon>0$ small the function~$U_{\epsilon}:=(p-\epsilon)u_{s}$ has a nonnegative fractional Laplacian while the function itself is sign-changing in~$E_a$.

We begin with a heuristic explanation of the strategy.
We choose~$p(x)=p_2(x_1)+\gamma p_1(x_1)-\delta q(x)$ for constants~$\gamma,\delta\geq0$ to be fixed later and where
\[
p_2(x_1)=(1-x_1)^2,\quad p_1(x_1)=1-x_1,\quad\text{and}\quad q(x)=\sum_{k=2}^{n-1}x_k^2+ax_n^2,
\qquad x\in\R^n.
\]
From Lemmas~\ref{polynomial-of-degree-one},~\ref{polynomial-of-degree-two}, and~\ref{polynomial-of-degree-twob} it follows that
\begin{align*}
\Ds(p_2u_{s})&=P_2(x_1)+R_2(x_2,\ldots,x_n)  & \text{for some degree~$2$ polynomials~$P_2$ and~$R_2$, }\\
\Ds(p_1u_{s})&=P_1(x_1), & \text{for some degree~$1$ polynomial~$P_1$, }\\
\Ds(qu_{s})&=Q(x_1)+R_0(x_2,\ldots,x_n)  & \text{for some degree~$2$ polynomials~$Q$ and~$R_0$. }
\end{align*}
To achieve~\eqref{goal-counter} we then need, in particular, that~$\delta$ satisfies
\begin{equation}\label{choice-of-delta}
R_2-\delta R_0\geq 0, \qquad \text{in }E_a.
\end{equation}
The choice of~$\gamma$ is far more delicate, but from a geometric point of view it can be made intuitively optimal:
indeed, in the worst case scenario, the polynomial~$P_2(x_1)$ has two real roots~$P_{2,-}< P_{2,+}<1$, while~$P_1(x_1)$ always has one~$P_{1,+}$. 
In this case, it holds that~$P_{2,+}$ and~$P_{1,+}$  are both of the order
\begin{align*}
1-P_{2,+}=O\Big(\frac1a\Big)=1-P_{1,+}, 
\qquad \text{as }a\uparrow\infty.
\end{align*}
But then, if we aim at having~$P_2(x_1)+\gamma P_1(x_1)>0$ in~$(-1,1)$,
it is enough to verify (see Figure~\ref{pic polynomials})
\begin{equation}\label{main-condition}
P_{2,+}< P_{1,+}
\end{equation}
and consequently choose 
\begin{equation}\label{choice-of-gamma}
\gamma=-\frac{P_2'(P_{2,+})}{P_1'}.
\end{equation}
(noticing that the derivative of~$P_1$ is a negative constant): with this choice of~$\gamma$,
we will have~$P_2(x_1)\geq -\gamma P_1(x_1)$ in~$(-1,1)$ by convexity.

By taking~$\delta>0$ such that~\eqref{choice-of-delta} is satisfied, and replacing~$P_2(x_1)$ with~$P_2(x_1)+\delta Q(x_1)$, the range of possible choices of~$s$ so that~\eqref{main-condition} is satisfied can even be enlarged.
\medskip

The conditions that need to be verified in this argument and their compatibility (on top of an asymptotic analysis as~$a\uparrow\infty$) is basically the technical reason why the strategy stops working at finite~$s$: nevertheless we expect that increasing the degrees of the involved polynomials could give some more flexibility in the computations, resulting in a wider range for~$s$. 

\setlength{\unitlength}{\textwidth}
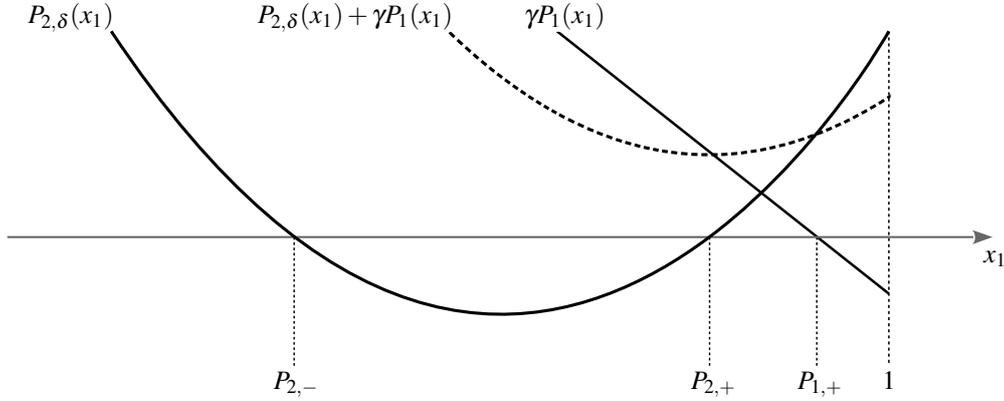
\begin{figure}[!hbt]
\begin{minipage}{\textwidth}
\begin{center}
\begin{picture}(.9,.4)
\put(0,0){\includegraphics[width=.9\textwidth]{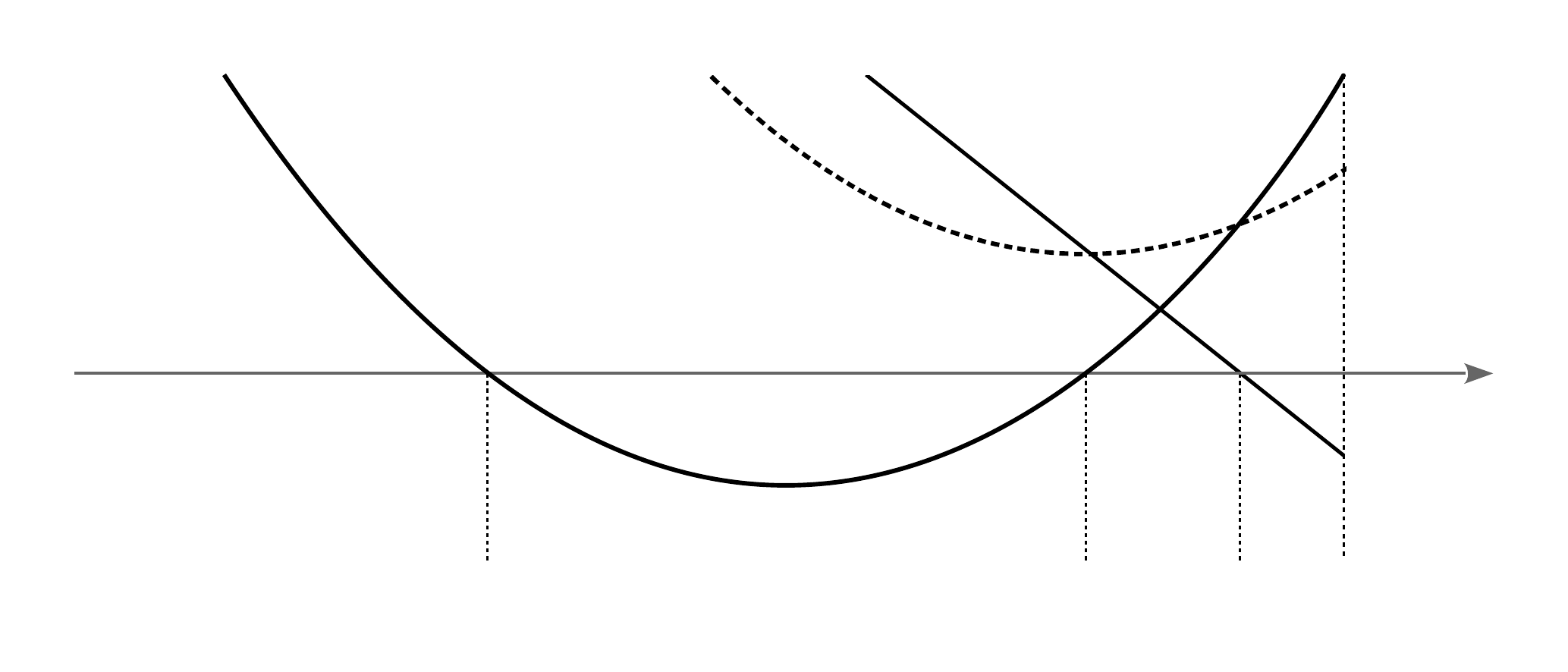}}
\put(.06,.35){$P_{2,\delta}(x_1)$}
\put(.25,.35){$P_{2,\delta}(x_1)+\gamma P_1(x_1)$}
\put(.47,.35){$\gamma P_1(x_1)$}
\put(.263,.045){$P_{2,-}$}
\put(.608,.045){$P_{2,+}$}
\put(.696,.045){$P_{1,+}$}
\put(.766,.045){$1$}
\put(.85,.15){$x_1$}
\end{picture}
\end{center}
\end{minipage}
\caption{A choice of~$\gamma$ that implies~$P_{2,\delta}+\gamma P_1>0$. }\label{pic polynomials}
\end{figure}

Theorem~\ref{thm:ce:ext} follows directly from the next result.
\begin{theorem}\label{counterexample1}
	Let
	\begin{align}\label{pg}
	 p(x):=(1-x_1)^2+\gamma(1-x_1)-\delta \Big(\sum\limits_{k=2}^{n-1}x_k^2+ax_n^2\Big).
	\end{align}
Then, for every~$s\in(1,\sqrt3+3/2)$, there are~$\gamma,\delta\geq 0$, and~$a_0>1$ such that the following holds: for every~$a\geq a_0$  there is~$K>0$ such that 
	\begin{align*}
	\Ds (p u_{s})(x)\geq K\,\frac{J_0}{a^2}\quad\text{ for all~$x\in E_a$.}
	\end{align*}
	In particular, for every~$a\geq a_0$ there is~$\epsilon>0$ such that the function~$U_{\epsilon}=(p-\epsilon)u_{s}\in\cH^s_0(E_a)$ satisfies
	\begin{align*}
	\Ds U_{\epsilon}(x)>0 \quad\text{ for all }x\in E_a.
	\end{align*}
\end{theorem}

\begin{proof}[Proof of Theorem~\ref{counterexample1}]
	In the following, we perform an asymptotic analysis letting~$a\uparrow\infty$. To this end, let us first recall~\eqref{J0} and~\eqref{Ji}.
	By Lemma~\ref{asymptotic-ji}, we have
		\begin{equation}\label{counterexample1:asymptotic1}
		j_1:=\lim\limits_{a\uparrow\infty} \frac{aJ_1^{(1)}}{J_0}=\frac{1}{2s-1}\quad\text{and}\quad j_2:=\lim\limits_{a\uparrow\infty} \frac{a^2J_2^{(1)}}{J_0}=\left\{\begin{aligned} &+\infty, &&\text{if}\  s\in\Big(1,\frac{3}{2}\Big];\\
		&\frac{3}{(2s-1)(2s-3)}=\frac{3j_1}{2s-3}, &&\text{if}\  s>\frac{3}{2}.
		\end{aligned}\right.
		\end{equation}
		Moreover,~$\lim\limits_{a\uparrow\infty} \frac{aJ_2^{(1)}}{J_0}=\lim\limits_{a\uparrow\infty} \frac{J_2^{(1)}}{J_1^{(1)}}=0$ for all~$s>1$. Let
		\begin{align}
        \begin{aligned}\label{ABC}
		A&:=(1-s\delta)J_0+s(5-\delta(2s-3))J_1^{(1)}+2s(s-1)(1+\delta)J_2^{(1)},\\
		B&:=J_0+2sJ_1^{(1)},\quad\text{and}\quad  C:=(1+\delta s)J_0-s(1+\delta)J_1^{(1)}.
		\end{aligned}
		\end{align}
		We denote by
		\begin{align}
		P_1(x_1)&=J_0-Bx_1,\notag\\
		P_{2,\delta}(x_1)&=Ax_1^2-2Bx_1+C,\notag \\
		Q_{\delta}(x_2,\ldots,x_n)&=s\sum_{i=2}^{n-1}\Big[J_1^{(1)}+2(s-1)\int_{\partial E_a}\theta_i^2\theta_1^2\;\mu(d\theta)\Big] x_i^2 
		+s\Big[J_1^{(1)}+2(s-1)a\int_{\partial E_a}\theta_n^2\theta_1^2\;\mu(d\theta)\Big]ax_n^2 \notag\\
&\quad -\delta \sum_{i=2}^{n-1}\Big[(s+1)J_0+4sJ_1^{(i)}-sJ_1^{(1)}+2s(s-1)J_2^{(i)}\notag \\
& \qquad\qquad +2s(s-1)\sum_{\overset{h=2}{h\neq i}}^{n-1} \int_{\partial E_a}\theta_h^2\theta_i^2\mu(d\theta)
+2s(s-1)a\int_{\partial E_a}\theta_n^2\theta_1^2\mu(d\theta)\Big]x_i^2\notag \\
& \quad -\delta\Big[(s+1)J_0+4sJ_1^{(n)}-sJ_1^{(1)} +2s(s-1)J_1^{(n)} -2s(s-1) a\int_{\partial E_a}\theta_n^2\theta_1^2\mu(d\theta)\Big]ax_n^2,\label{Qdelta}
		\end{align}
		so that, for~$x\in E_a$, we have
		\[
		\frac{\Ds (pU)(x)}{\Gamma(1+s)k_{n,s}}=P_{2,\delta}(x_1)+\gamma P_1(x_1) + Q_{\delta}(x_2,\ldots,x_n),
		\qquad x\in E_a.
		\]
		We first note that the discriminant of~$P_{2,\delta}$ is given by
		\begin{align*}
		&B^2-AC=J_0^2+4sJ_0J_1^{(1)}+4s^2(J_1^{(1)})^2\\
		&\qquad -\big((1-s\delta)J_0+s(5-\delta(2s-3))J_1^{(1)}+2s(s-1)(1+\delta)J_2^{(1)}\big)(J_0(1+s\delta)-s(1+\delta)J_1^{(1)})\\
		&=J_0^2+4sJ_0J_1^{(1)}+4s^2(J_1^{(1)})^2 -(1-s^2\delta^2)J_0^2-s(1+s\delta)(5-\delta(2s-3))J_0J_1^{(1)}-2s(s-1)(1+\delta)(1+s\delta)J_0J_2^{(1)}\\
		&\qquad +s(1+\delta)(1-s\delta)J_0J_1^{(1)}+s^2(5-\delta(2s-3))(1+\delta)(J_1^{(1)})^2+2s^2(s-1)(1+\delta)^2J_1^{(1)}J_2^{(1)}\\
		&=s^2\delta^2 J_0^2+s\Big(4 -(1+s\delta)(5-\delta(2s-3))+(1+\delta)(1-s\delta)\Big)J_0J_1^{(1)}+s^2\Big(4+(5-\delta(2s-3))(1+\delta)\Big)(J_1^{(1)})^2\\
		&\qquad  -2s(s-1)(1+\delta)(1+s\delta)J_0J_2^{(1)}+2s^2(s-1)(1+\delta)^2J_1^{(1)}J_2^{(1)}\\
		&=s^2\delta^2 J_0^2-2s\delta\Big(2s+1+s(2-s)\delta\Big)J_0J_1^{(1)}+s^2\Big(9+(8-2s)\delta-(2s-3)\delta^2\Big)(J_1^{(1)})^2\\
		&\qquad  -2s(s-1)(1+\delta)(1+s\delta)J_0J_2^{(1)}+2s^2(s-1)(1+\delta)^2J_1^{(1)}J_2^{(1)}.
		\end{align*}
		If~$s\in(1,3/2]$ and~$\delta=0$, then
		\[
		\frac{a^2(B^2-AC)}{J_0^2}= 9s^2\Big(\frac{aJ_1^{(1)}}{J_0}\Big)^2  -2s(s-1)\frac{a^2J_2^{(1)}}{J_0}+2s^2(s-1)\frac{aJ_1^{(1)}}{J_0}\frac{aJ_2^{(1)}}{J_0}\downarrow-\infty,
		\qquad \text{as }a\uparrow\infty,
		\]
		so that there is~$a_0>0$ such that~$P_{2,0}$ is positive for all~$a\geq a_0$. On the other hand, if~$s\in(3/2,2)$ and~$\delta=0$, then, using~\eqref{counterexample1:asymptotic1},
		\begin{multline*}
		\frac{a^2(B^2-AC)}{J_0^2} = 9s^2\Big(\frac{aJ_1^{(1)}}{J_0}\Big)^2-2s(s-1)\frac{a^2J_2^{(1)}}{J_0}+2s^2(s-1)\frac{aJ_1^{(1)}}{J_0}\frac{aJ_2^{(1)}}{J_0} \\
		\longrightarrow 9s^2j_1^2-\frac{6s(s-1)}{2s-3}j_1=3sj_1\Big(\frac{3s}{2s-1}-\frac{2(s-1)}{2s-3}\Big) 
		=\frac{3 s j_1}{(2s-1)(2s-3)}(s-2)(2s+1)<0 \qquad \text{as }a\uparrow\infty.
		\end{multline*}
		The claim in the case~$s\in(1,2)$ hence follows by choosing~$\delta=\gamma=0$, noting that $Q_0\geq0$ since it is the sum of nonnegative terms.
				
		In the following we assume~$s\geq 2$. Moreover, we assume that~$\delta$ is such that
		\begin{equation}\label{small-condition}
		A=(1-s\delta)J_0+s(5-\delta(2s-3))J_1^{(1)}+2s(s-1)(1+\delta)J_2^{(1)}>0:
		\end{equation}
		this is asymptotically satisfied as~$a\uparrow\infty$ if~$s\delta<1$.
		
		For the positivity of~$Q_\delta$ first note that, by symmetry,~$J_i^{(k)}=J_i^{(1)}$ for~$k\in\{1,\ldots,n-1\}$ and~$i\in\N$; furthermore,
		\begin{align}\label{31287}
		J_1^{(n)}  = J_0-(n-1)J_1^{(1)},
		\quad \int_{\partial E_a}\theta_1^2\theta_k^2\,\mu(d\theta) = \frac{1}{3}J_2^{(1)},\quad\text{and}\quad a\int_{\partial E_a}\theta_n^2\theta_1^2\,\mu(d\theta)  = J_1^{(1)}-\frac{n+1}{3} J_2^{(1)},
		\end{align}
		where the last two identities follow from Lemma~\ref{identities-ik} and the first identity is a consequence of the definition of~$J_1^{(n)}$ and of~$E_a$. Hence, again by symmetry, the fact that $\theta_2^2+\ldots+\theta_{n-1}^2+a\theta_n^2=1-\theta_1^2$ for $\theta\in \partial E_a$, and~\eqref{31287}, 
\begin{align*}
%Q_{\delta}(x_2,\ldots,x_n)= & 
%\sum_{i=2}^{n-1}x_i^2
%\Bigg[
%sJ_1^{(1)}+2s(s-1)\underbrace{\int_{\partial E_a}\theta_2^2\theta_1^2\;\mu(d\theta)}_{\frac{1}{3}J_2^{(1)}}-\delta\Big[(s+1)J_0+3sJ_1^{(1)}+2s(s-1)J_2^{(1)} \\
%& \qquad\qquad +2s(s-1)  \int_{\partial E_a}\Big(\underbrace{\theta_2^2+\ldots+\theta_{n-1}^2+a\theta_n^2}_{1-\theta_1^2}\Big)\theta_i^2\;\mu(d\theta)-2s(s-1)\underbrace{J_2^{(i)}}_{J_2^{(1)}}\Big]
%\Bigg]\\
Q_{\delta}(x_2,\ldots,x_n)= & 
\sum_{i=2}^{n-1}x_i^2
\Bigg[
sJ_1^{(1)}+2s(s-1)\int_{\partial E_a}\theta_2^2\theta_1^2\;\mu(d\theta)-\delta\Big[(s+1)J_0+3sJ_1^{(1)}+2s(s-1)J_2^{(1)} \\
& \qquad\qquad +2s(s-1)  \int_{\partial E_a}\Big( \theta_2^2+\ldots+\theta_{n-1}^2+a\theta_n^2 \Big)\theta_i^2\;\mu(d\theta)-2s(s-1)J_2^{(i)}\Big]
\Bigg]\\
&+ax_n^2\Bigg[sJ_1^{(1)}+2s(s-1)a\int_{\partial E_a}\theta_n^2\theta_1^2\;\mu(d\theta)\\
&\qquad\qquad-\delta\Big((s+1)J_0+4sJ_1^{(n)}-sJ_1^{(1)} +2s(s-1)J_1^{(n)} -2s(s-1) a\int_{\partial E_a}\theta_n^2\theta_1^2\;\mu(d\theta)\Big)\Bigg]\\
=& \sum_{i=2}^{n-1}x_i^2\Bigg[sJ_1^{(1)}+\frac{2}{3}s(s-1)J_{2}^{(1)}-\delta\Big((s+1)J_0+3sJ_1^{(1)}+2s(s-1)  \int_{\partial E_a}(1-\theta_1^2)\theta_2^2\;\mu(d\theta) \Big)\Bigg]\\
&+ax_n^2\Bigg[sJ_1^{(1)}+2s(s-1) \Big(J_1^{(1)}-\frac{n+1}{3} J_2^{(1)}\Big)\\
&\qquad\qquad-\delta\Big((s+1)J_0+2s(s+1)J_1^{(n)}-sJ_1^{(1)} -2s(s-1)(J_1^{(1)}-\frac{n+1}{3} J_2^{(1)})\Big)\Bigg]\\
=& \sum_{i=2}^{n-1}x_i^2\Bigg[sJ_1^{(1)}+\frac{2}{3}s(s-1)J_{2}^{(1)}-\delta\Big((s+1)J_0+s(2s+1)J_1^{(1)}-\frac{2}{3}s(s-1)J_{2}^{(1)}\Big)\Bigg]\\
&+ax_n^2\Bigg[s(2s-1)J_1^{(1)}-2s(s-1)\frac{n+1}{3} J_2^{(1)} \\
&\qquad\qquad-\delta\Big((2s+1)(s+1)J_0-s\big[2(s+1)(n-1)+2s-1\big]J_1^{(1)}+2s(s-1)\frac{n+1}{3} J_2^{(1)}\Big)\Bigg].
\end{align*}
This combined with the asymptotic estimates in Lemma~\ref{asymptotic-ji} gives~$Q_{\delta}\geq0$ for~$a$ sufficiently large, if
\begin{align*}
sJ_1^{(1)}-\delta(s+1)J_0>0\quad\text{and}\quad s(2s-1)J_1^{(1)}-\delta(s+1)(2s+1)J_0 > 0.
\end{align*}
Note that the second inequality implies the first and in view of the last inequality, we choose
\begin{align}\label{deltafirst}
\delta=O\Big(\frac1a\Big)
\qquad\text{and}\qquad
\delta<\frac1a\lim_{a\uparrow\infty}\frac{s(2s-1)\, aJ_1^{(1)}}{(s+1)(2s+1)J_0}=\frac{s}{(s+1)(2s+1)}\frac1a;
\end{align}
remark how this choice for~$\delta$ also fulfills~\eqref{small-condition} for~$a$ large.

		Note that, in view of~\eqref{small-condition}, the largest root of~$P_{2,\delta}$ is given by
		\begin{equation}\label{right-root-p2}
		P_{2,+}:=\frac{B+\sqrt{B^2-AC}}{A},
		\end{equation}
		provided\footnote{If this is not the case, then~$P_{2,\delta}$ is positive and it is sufficient to take~$\gamma=0$.}~$B^2\geq AC$.
		We remark that\footnote{We use the asymptotic behaviors stated in~\eqref{small-ji}, on top of identities~\eqref{ik} and~\eqref{in}: mind that all this relies on the restriction~$s>3/2$.}
		\begin{align*}
		B^2-AC & = s^2\delta^2 J_0^2-2s\delta(2s+1)J_0J_1^{(1)}+9s^2(J_1^{(1)})^2-2s(s-1)J_0J_2^{(1)} \\
		& = J_0^2\Big(s^2\delta^2-\frac{2s(2s+1)}{2s-1}\frac\delta{a}+\frac{9s^2}{(2s-1)^2}\frac1{a^2}-\frac{6s(s-1)}{(2s-1)(2s-3)}\frac1{a^2}\Big)+o\Big(\frac{J_0^2}{a^2}\Big),
		\qquad \text{as }a\uparrow\infty.
		\end{align*}
		The root of~$P_1$ is given by
		\begin{equation}\label{root-p1}
		P_{1,+}:=\frac{J_0}{B}.
		\end{equation}
		As explained above, with~$\gamma$ as in~\eqref{choice-of-gamma} we have~$P_{2,\delta}+\gamma P_1>0$ in~$[-1,1]$, if (and only if) we can find~$\delta$ such that
		\begin{align}\label{plp}
		P_{2,+}<P_{1,+},
		\end{align}
where the strict inequality is needed due to the asymptotic analysis. This inequality is moreover equivalent to
		\begin{align*}
		B^2+B\sqrt{B^2-AC}&<J_0A.
		\end{align*}
		Asymptotically, this is satisfied if and only if 
		\begin{align*}
		1+\frac{4s}{2s-1}\frac1a + \sqrt{s^2\delta^2-\frac{2s(2s+1)}{2s-1}\frac\delta{a}+\frac{9s^2}{(2s-1)^2}\frac1{a^2}-\frac{6s(s-1)}{(2s-1)(2s-3)}\frac1{a^2}} < 1-s\delta+\frac{5s}{2s-1}\frac1a,
		\end{align*}
		which is equivalent to
		\begin{align*}
		s^2\delta^2-\frac{2s(2s+1)}{2s-1}\frac\delta{a}+\frac{9s^2}{(2s-1)^2}\frac1{a^2}-\frac{6s(s-1)}{(2s-1)(2s-3)}\frac1{a^2} < \Big(-s\delta+\frac{s}{2s-1}\frac1a\Big)^2
		\qquad \text{for }\delta<\frac1{2s-1}\frac1a,
		\end{align*}
		\textit{i.e.},
% 		\begin{align*}
% 		\frac{2s^2-2s(2s+1)}{2s-1}\frac\delta{a}<\Big(\frac{6s(s-1)}{(2s-1)(2s-3)}-\frac{8s^2}{(2s-1)^2}\Big)\frac1{a^2}
% 		\qquad \text{for }\delta<\frac1{2s-1}\frac1a,
% 		\end{align*}
% 		or also
% 		\begin{align*}
% 		-\frac{s+1}{2s-1}\delta<\Big(\frac{3(s-1)}{(2s-1)(2s-3)}-\frac{4s}{(2s-1)^2}\Big)\frac1{a}
% 		\qquad \text{for }\delta<\frac1{2s-1}\frac1a,
% 		\end{align*}
% 		which is
		\begin{align}\label{deltasecond}
		\delta>-\frac1{s+1}\Big(\frac{3(s-1)}{2s-3}-\frac{4s}{2s-1}\Big)\frac1{a}
		\qquad \text{for }\delta<\frac1{2s-1}\frac1a.
		\end{align}
		As the condition~$a\delta<1/(2s-1)$ is already implied by~\eqref{deltafirst},
		we are left to verify what values of~$s$ allow for a non-empty range of~$\delta$ 
		as resulting from~\eqref{deltafirst} and~\eqref{deltasecond}: these are those values that satisfy
		\begin{align*}
		-\frac{3(s-1)}{2s-3}+\frac{4s}{2s-1}<\frac{s}{2s+1},
		\end{align*}
		which in particular holds for~$s\in[2,\sqrt3+3/2)$. % where~$s_0$ is as in the statement.
\end{proof}

\begin{proof}[Proof of Theorem~\ref{ce:intro:thm}]
This follows directly from the first part of the proof of Theorem~\ref{counterexample1}. 
\end{proof}

\subsection{A computer-assisted analysis in two dimensions}\label{ca:sec}

Theorem~\ref{counterexample1} shows that the fractional Laplacian~$(-\Delta)^s$ does not satisfy a positivity preserving property in the ellipse~$E_a$ for~$a$ large enough. Its proof uses an asymptotic analysis as~$a\uparrow\infty$ and constructs an explicit counterexample for any~$a$ sufficiently large ($a>a_0$ for some~$a_0>1$) 
and for~$s\in(1,s_0)$ with~$s_0:=\sqrt3+3/2\approx 3.232$.  In this section we fix~$n=2$ and address the following questions:
\begin{enumerate}
 \item [$i)$] How large is~$a_0$?
 \item [$ii)$] What can be said for~$s\geq s_0$?
\end{enumerate}

The answer to these questions depends on the explicit calculations developed in Section~\ref{sec2}, which involve several hypergeometric functions.  These functions can be expressed as a series~\eqref{s} or as an integral~\eqref{i}. However, direct calculations using these representations are usually hard to perform; nevertheless, computers are very efficient and precise manipulating and approximating the values of hypergeometric functions, and we use this to answer questions~$i)$ and~$ii)$. 

\subsubsection{The behaviour of \texorpdfstring{$a_0$}{a0} in the simplest case}\label{g:r}
Let 
\begin{align}\label{p}
p(x):=(1-x_1)^2,\qquad x\in\R^2 
\end{align}
then the value of~$\Ds (p u_{s})$ in~$E_a$ can be computed explicitly in terms of hypergeometric functions, see Table~\ref{table:some more examples}. In particular, 
\begin{align}\label{g}
\Ds (p u_{s})>0\quad \text{ in~$E_a$ \qquad if } B^2-AC<0,
\end{align}
where~$A$,~$B$, and~$C$ are given in~\eqref{ABC}.  In Figure~\ref{f2} we present a plot of the nodal regions of~$D(a,s):=B^2-AC$ (note that~$A$,~$B$, and~$C$ are all explicit functions of~$a$ and~$s$). 

%\begin{figure}[!htb]
\begin{center}
\setlength{\unitlength}{1cm}
\includegraphics[width=.9\textwidth]{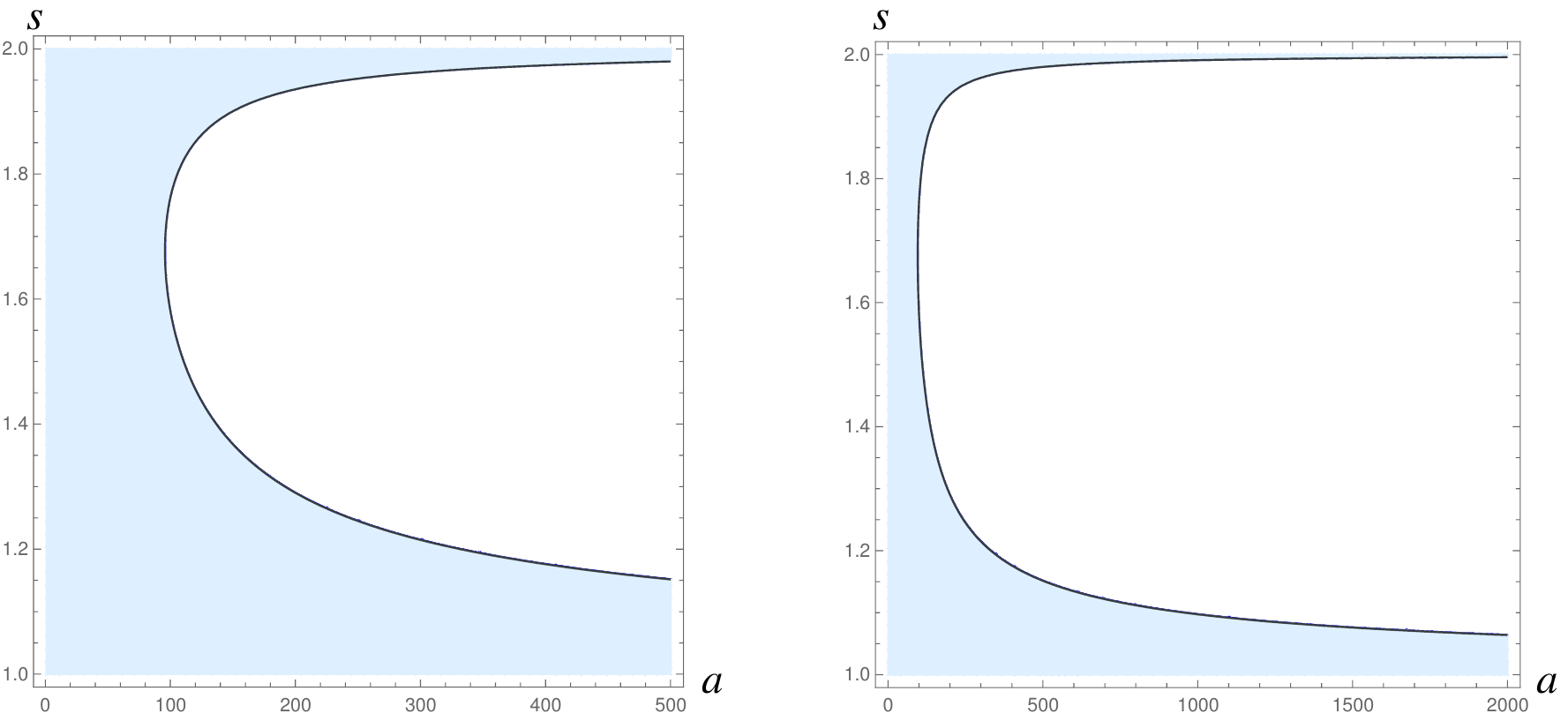}
% \begin{picture}(13,6)
% % \put(3.5,3.5){$D^-$}
% % \put(10,3.5){$D^-$}
% \put(0.2,5.6){$s$}
% \put(5.7,.2){$a$}
% \put(7.1,5.6){$s$}
% \put(12.5,.2){$a$}
% \includegraphics[width=.35\textwidth]{dtop1.pdf}\hspace{1.2cm}
% \includegraphics[width=.35\textwidth]{dtop2.pdf}
% \end{picture}
%  \includegraphics[width=.35\textwidth]{dtop1.pdf}\hspace{1.5cm}
%  \includegraphics[width=.35\textwidth]{dtop2.pdf}
\captionof{figure}{The nodal regions of~$D(a,s):=B^2-AC$ for~$s\in(1,2)$ with~$a\in(1,500)$ (left) and~$a\in(1,2000)$ (right). For~$(s,a)$ in the white region one can construct a counterexample to positivity preserving properties for~$(-\Delta)^s$ in the ellipse~$E_a$ with axes~$1$ and~$\frac{1}{\sqrt{a}}$.}\label{f2}
\end{center}
%\end{figure}

In particular, Figure~\ref{f2} shows that~\eqref{g} holds for all~$s\in(1,2)$ and~$a>a_0$ for some~$a_0>0$, as stated in Theorem~\ref{counterexample1}, however~$a_0\uparrow \infty$ as~$s\downarrow 1$, whereas for~$s=3/2$ we have~$a_0<115$.  Note that, if~$s\uparrow 2$, then we also have that~$a_0\uparrow \infty$ whenever~$p$ has the simple form~\eqref{p}; but, by using a more general polynomial~$p$ as in~\eqref{pg} for suitable~$\delta$ and~$\gamma$, one can obtain a counterexample for~$s$ larger.

\subsubsection{Extended range for counterexamples}

If~$s\geq s_0$, then the asymptotic analysis in the proof of Theorem~\ref{counterexample1} cannot be successfully implemented.  However, one can show that a counterexample can be obtained for some~$s\geq s_0$ if~$a$ is \textit{not} very large.

To be more precise, let~$\gamma$ be as in~\eqref{choice-of-gamma} and let
\begin{align*}
 \delta &= \frac{s \big(J_1^{(1)}+2 (s-1) (J_1^{(1)}-J_2^{(1)})\big)}{(s+1) J_0-s J_1^{(1)}-2s(s-1) 
   (J_1^{(1)}-J_2^{(1)})+2s(s+1)J_1^{(2)}
   %\frac{\pi  (2 (s-1) s+4 s) \,_2F_1(\frac{3}{2},s+1;2;1-\frac{1}{a})}{\sqrt{a}}
   } \\
   &=\frac{s (a-1) \big(\, _2F_1(\frac{1}{2},s+1;1;\frac{a-1}{a})-\,
   _2F_1(\frac{1}{2},s+1;2;\frac{a-1}{a})\big)}{
   _2F_1\left(\frac{1}{2},s+1;1;\frac{a-1}{a}\right)+\big((a-1) s+a-2\big) \,
   _2F_1(\frac{3}{2},s+1;1;\frac{a-1}{a})}. 
\end{align*}
This choice of~$\delta$ is such that~$Q_\delta\equiv 0$ (see~\eqref{Qdelta} and use~\eqref{explicit-ji} and~\eqref{in}).

Let~$P_{1,+}$ and~$P_{2,+}$ be as in~\eqref{right-root-p2} and~\eqref{root-p1}.  Then a counterexample can be successfully constructed if~$P_{1,+}>P_{2,+}$, see~\eqref{plp}. Let 
\begin{align*}
 h(a,s):=P_{1,+}-P_{2,+}.
\end{align*}
Then we can compute numerically that~$h(11,s)>0$ for~$s\in [3,3.8456)$, see Figure~\ref{f3}.  Observe also that~$h(20,3.8)<0$; in particular, this implies that large values of~$a$ are not always optimal to construct a counterexample. 

To argue the optimality and the consistency of our approach, we remark that the root of the mapping~$a\mapsto h(a,2)$ can be computed numerically, and it is given by~$b_0\approx 18.94281916344395$ (see Figure~\ref{f3}), which is the same threshold found in~\cite[Theorem 5.2]{rg12}, obtained with different arguments than ours in the study of the bilaplacian in two-dimensional ellipses.

%\begin{figure}[!htb]
\begin{center}
 \setlength{\unitlength}{1cm}
 \includegraphics[width=.9\textwidth]{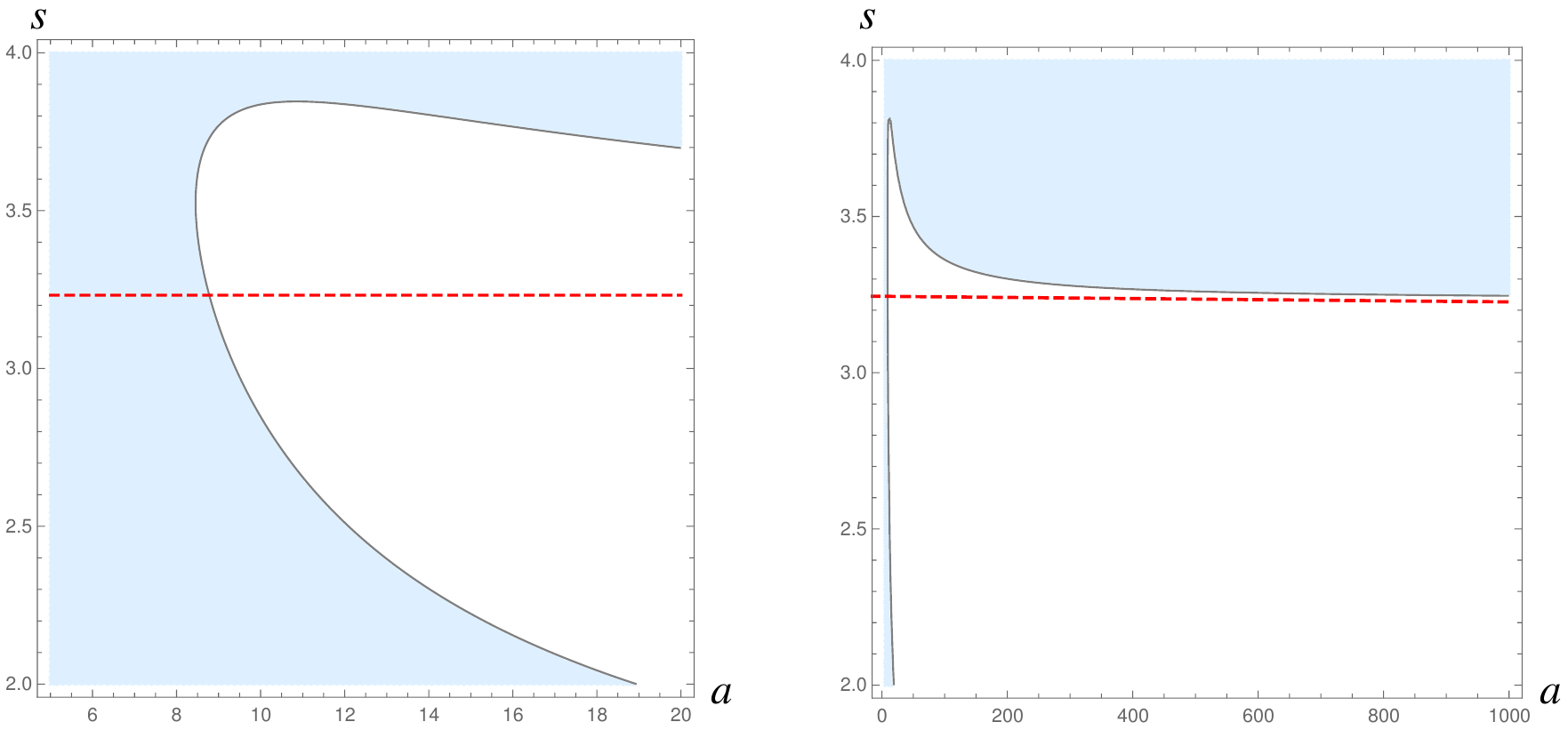}
% \begin{picture}(13,6)
% % \put(3.5,3){$h^+$}
% % \put(9.5,2.7){$h^+$}
% \put(0.2,5.65){$s$}
% \put(5.7,.2){$a$}
% \put(6.9,5.65){$s$}
% \put(12.4,.2){$a$}
%  \includegraphics[width=.35\textwidth]{4.pdf} \hspace{1cm}
%  \includegraphics[width=.35\textwidth]{3.pdf}
% \end{picture}
\captionof{figure}{The nodal regions of~$h(a,s)$ for~$(a,s)\in(5,20)\times(2,4)$ (left) and for~$(a,s)\in(5,1000)\times(2,4)$ (right). The dashed line represents~$s_0=\sqrt{3}+3/2$. For~$(s,a)$ in the white region one can construct a counterexample to positivity preserving properties for~$\Ds$ in the ellipse~$E_a$ with axes~$1$ and~$\frac{1}{\sqrt{a}}$.}\label{f3}
\end{center}
%\end{figure}

\section{Point inversion transformations}\label{shapes}

For~$a,c>0$ and~$\nu\in \R^n$, let~$\sigma$ and~$\Omega=\Omega(a,c,\nu)$ be defined as in Corollary~\ref{omega:cor}, namely, for~$x\in \R^n\setminus\{-\nu\},$
\begin{align*}
 \sigma (x):=c\,\frac{x+\nu}{|x+\nu|^{2}}-\nu,
 \qquad  \Omega=\Omega(a,c,\nu):=\left\{x\in\R^n\::\: \sum_{i=1}^{n-1} \sigma_i(x)^2+a\sigma_n(x)^2<1\right\}.
\end{align*}
The geometrical meaning of the point inversion transformation~$\sigma$ is that of an inversion with respect to the boundary of a sphere of radius~$\sqrt{c}$ centered in~$-\nu$, see Figure~\ref{pinv}. Note that if~$c=1$ and~$\nu=0$, then~$\sigma$ is the usual Kelvin transform.

%\begin{figure}[!htb]
\begin{center}
\includegraphics[width=.29\textwidth]{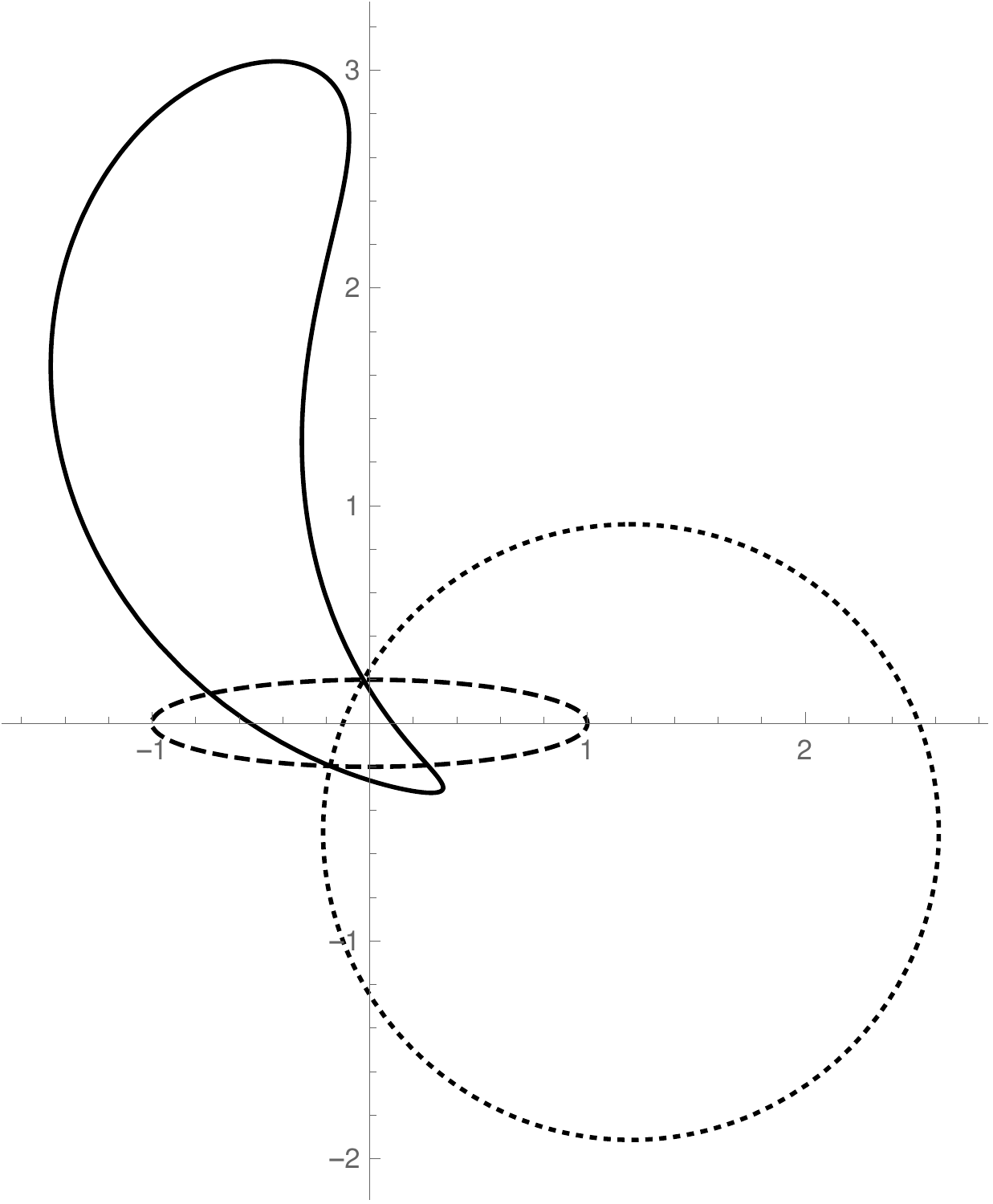}
\qquad
\includegraphics[width=.36\textwidth]{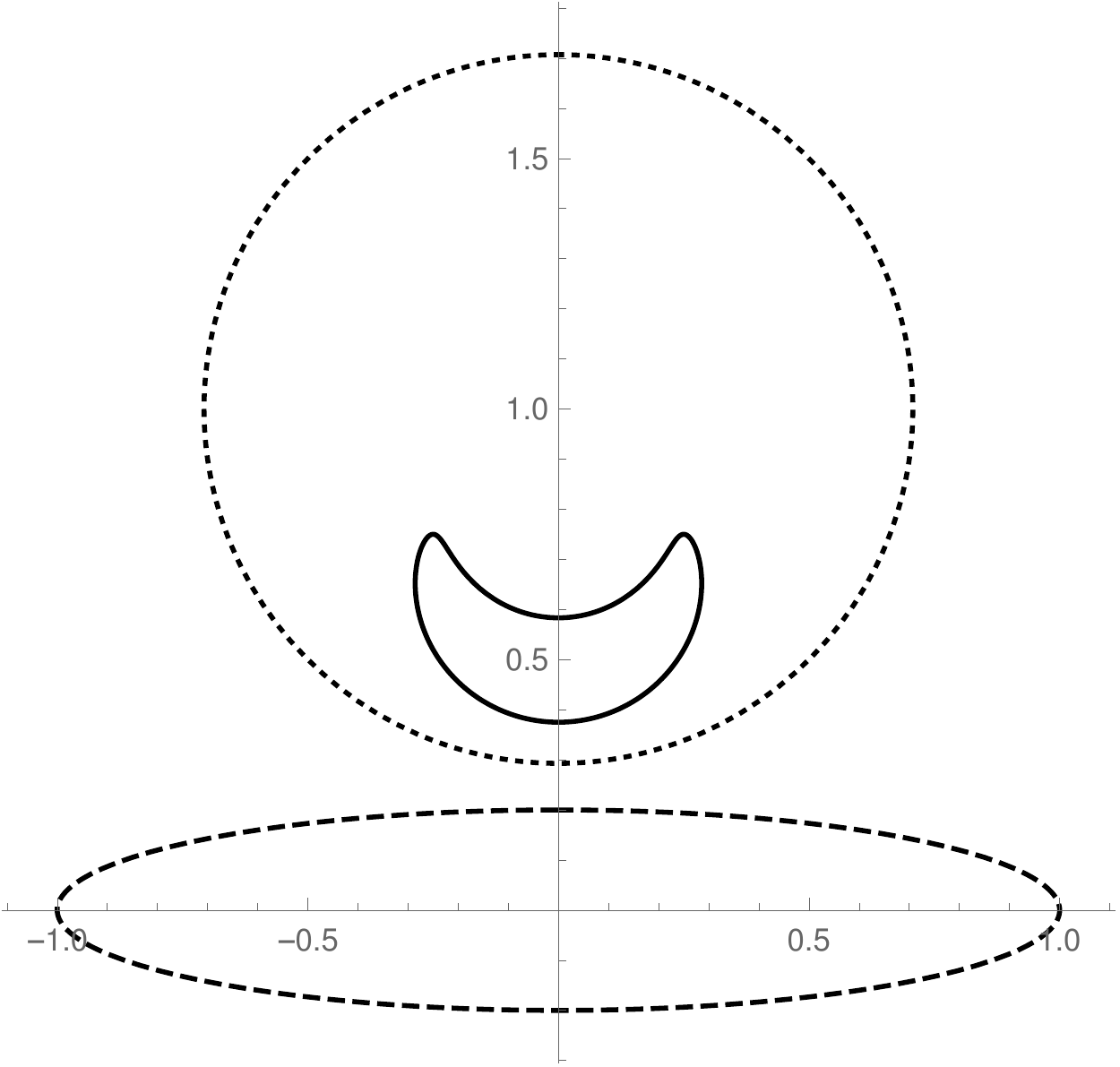}
\captionof{figure}{The associated point inversion~$\sigma$ is an inversion with respect to the boundary~$\partial B_{\sqrt{c}}(-\nu)$.  In the picture we see the ellipse~$E_{25}$ with axis of length 1 and~$\frac{1}{5}$ (dashed), its transformation~$\sigma(E_{25})$ and the circle~$\partial B_{\sqrt{c}}(-\nu)$ (dotted) for~$\nu=(-\frac{6}{5},\frac{1}{2})$,~$c=2$ (left), and for~$\nu=(0,-1)$,~$c=\frac{1}{2}$ (right). }\label{pinv}
\end{center}
%\end{figure}

Varying~$\nu$ and~$c$ gives rise to a wide variety of shapes, as illustrated in Figures~\ref{fig1} and~\ref{fig2} below. See also~\cite{gs1}, where a point inversion transformation is used to show the existence of domains for which the bilaplacian's torsion function is sign-changing. We thank G. Sweers for sharing references~\cite{gs1,gs2,dw05} with us.

\begin{proof}[Proof of Corollary~\ref{omega:cor}]
We argue as in~\cite[Proposition 1.6]{dcds}.  Fix~$c>0$,~$a\in \R^n$ with~$a_i>0$,~$\nu\in \R^n\setminus \partial E_a$,~$\sigma$ as in~\eqref{kelvin-trafo},~$\Omega:=\Omega(a,c,\nu)=\sigma(E_a)$, and let 
$K_{s} z(x):=|x+v|^{2s-n} z(\sigma (x))$ for~$x\in\R^n\backslash\{-v\}$ and~$z\in C(\R^n)$. Note that~$-\nu\not\in\overline{\Omega}$. Then, if~$u_s(x):=(1-\sum_{i=1}^{n}a_ix_i^2)^s_+$ we have that~$w_s=K_s u_s$. By~\cite[Lemma 3.3]{dcds}, one can compute~$(-\Delta)^s w_s$ pointwisely in~$\Omega$. Then, for every~$\phi\in C^\infty_c(\Omega)$,
\begin{align*}
& \int_{\Omega}w_s\,\Ds\phi\ =\ c^{2n-4s}\int_{\Omega}K_s(K_s w_s)(x)\,\Ds K_s(K_s\phi)(x)\ dx \\
& =c^{2n-2s}\int_{\Omega}\frac{K_s w_s(\sigma (x))}{|x+v|^{n-2s}}\,\frac{\Ds K_s\phi(\sigma (x))}{|x+\nu|^{n+2s}}\ dx\ 
=\ c^{n-2s}\int_{E_a}\,K_s w_s(y)\,\Ds K_s\phi(y)\ dy
=\int_{E_a} u_s\Ds K_s\phi,
\end{align*}
by a change of variables ($y=\sigma(x)$) and by Proposition~\ref{lem:sharm}, where we used that~$K_s(K_s z)=c^{2s-n} z$ and that the Jacobian for~$x\mapsto \sigma(x)$ is~$c^n|x+v|^{-2n}$.  Integrating by parts (see, for example,~\cite[Lemma 1.5]{cpaa}),
\begin{align*}
 \int_{\Omega}\,\Ds w_s(x) \phi(x)\ dx\ 
 =\int_{E_a} \,\Ds u_s(y) \frac{\phi(\sigma(y))}{|x+\nu|^{n-2s}}\ dy
 =\kappa\int_{\Omega} \frac{c^{n}|x+\nu|^{-2n}}{|\sigma(x)+\nu|^{n-2s}}\phi(x)\ dy
 =\int_{\Omega} \frac{k \phi(x)}{|x+\nu|^{n+2s}} \ dy
\end{align*}
for some constant~$k>0$. Since this holds for any~$\phi\in C^\infty_c(\Omega)$, we have that~$\Ds w_s(x)=k|x+\nu|^{-n-2s}$ pointwisely in~$\Omega$, as claimed.
\end{proof}

\begin{proof}[Proof of Corollary~\ref{omega:cor2}]
We use the notation from the proof of Corollary~\ref{omega:cor}.
Assume that~$\Omega$ is bounded or that~$n>4s$ and let~$U_\eps$ be given by Theorem~\ref{thm:ce:ext}. Then, a direct calculation shows that~$W:=K_s U_\eps\in L^2(\R^n)$.  Moreover,~$W$ is sign changing and, by Proposition~\ref{lem:sharm} and Plancherel's Theorem,
\begin{align*}
 \int_{\R^n}|\xi|^{2s}|\widehat W|^2
 &=\int_{\Omega} W (-\Delta)^{s} W
 =c^{2s}\int_{\Omega} \frac{U_\eps(\sigma(x))}{|x+\nu|^{n-2s}} \frac{(-\Delta)^{s} U_\eps(\sigma(x))}{|x+\nu|^{2s+n}}
 =c^{2s}\int_{\Omega} \frac{U_\eps(\sigma(x))}{|x+\nu|^{n-2s}} \frac{P(\sigma(x))}{|x+\nu|^{2s+n}}\\
 &=c^{n+2s}\int_{E_a} \frac{U_\eps(x)}{|\sigma(x)+\nu|^{n-2s}} \frac{P(x)}{|\sigma(x)+\nu|^{2s+n}}|x+\nu|^{-2n}
=c^{n+2s}\int_{E_a} U_\eps(x)P(x)<\infty,
\end{align*}
where~$\widehat W$ is the Fourier transform of~$W$ and~$P$ is a polynomial of degree two given by Lemmas~\ref{polynomial-of-degree-one} and~\ref{polynomial-of-degree-two}.  In particular~$W\in\cH^s_0(\R^n)$.  Arguing as in Corollary~\ref{omega:cor}, we obtain that~$(-\Delta)^s W>0$ pointwisely in~$\Omega$. 

\end{proof}
\begin{figure}[!htb]
\begin{center}
\includegraphics[width=\textwidth]{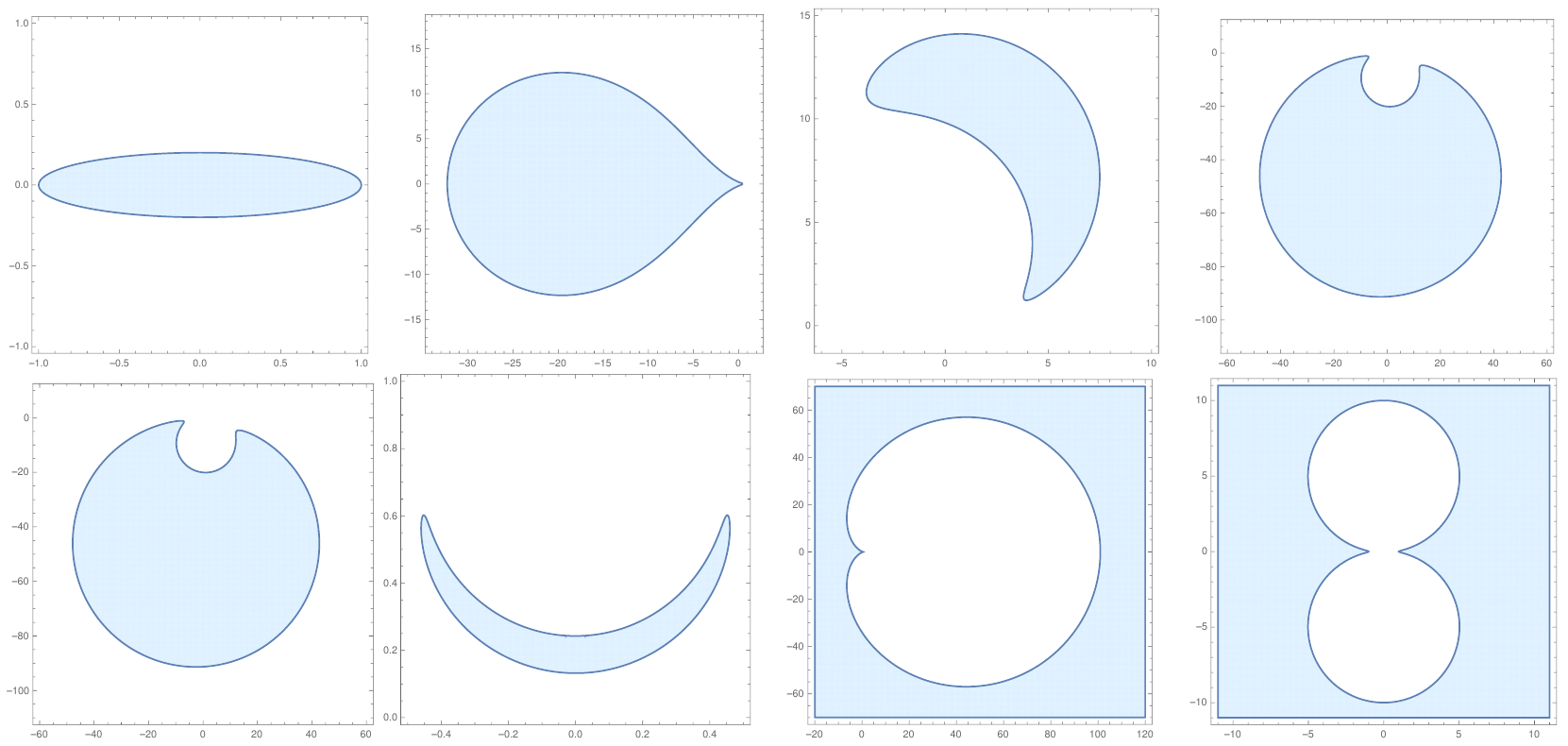}
\captionof{figure}{From left to right and top to bottom: the ellipse~$E_{25}$, $\Omega(5^2,1,(-1.03,0))$, $\Omega(5^2,10,(\frac{4}{5},\frac{4}{5}))$, $\Omega(5^2,10,(-\frac{3}{10},-\frac{3}{10}))$, $\Omega(15^2,1,(0,-1.1))$, and~$\Omega(5^2,\frac{1}{10},(-\frac{99}{100},0))$. The last two figures are~$\Omega(10^2,10,(0,0,-1.1))$ and~$\Omega(30^2,\frac{1}{10},(-0.99,0))$, which are examples of \textit{unbounded} domains.}\label{fig1}
\end{center}
\end{figure}

\begin{figure}
\begin{center}
\includegraphics[width=\textwidth]{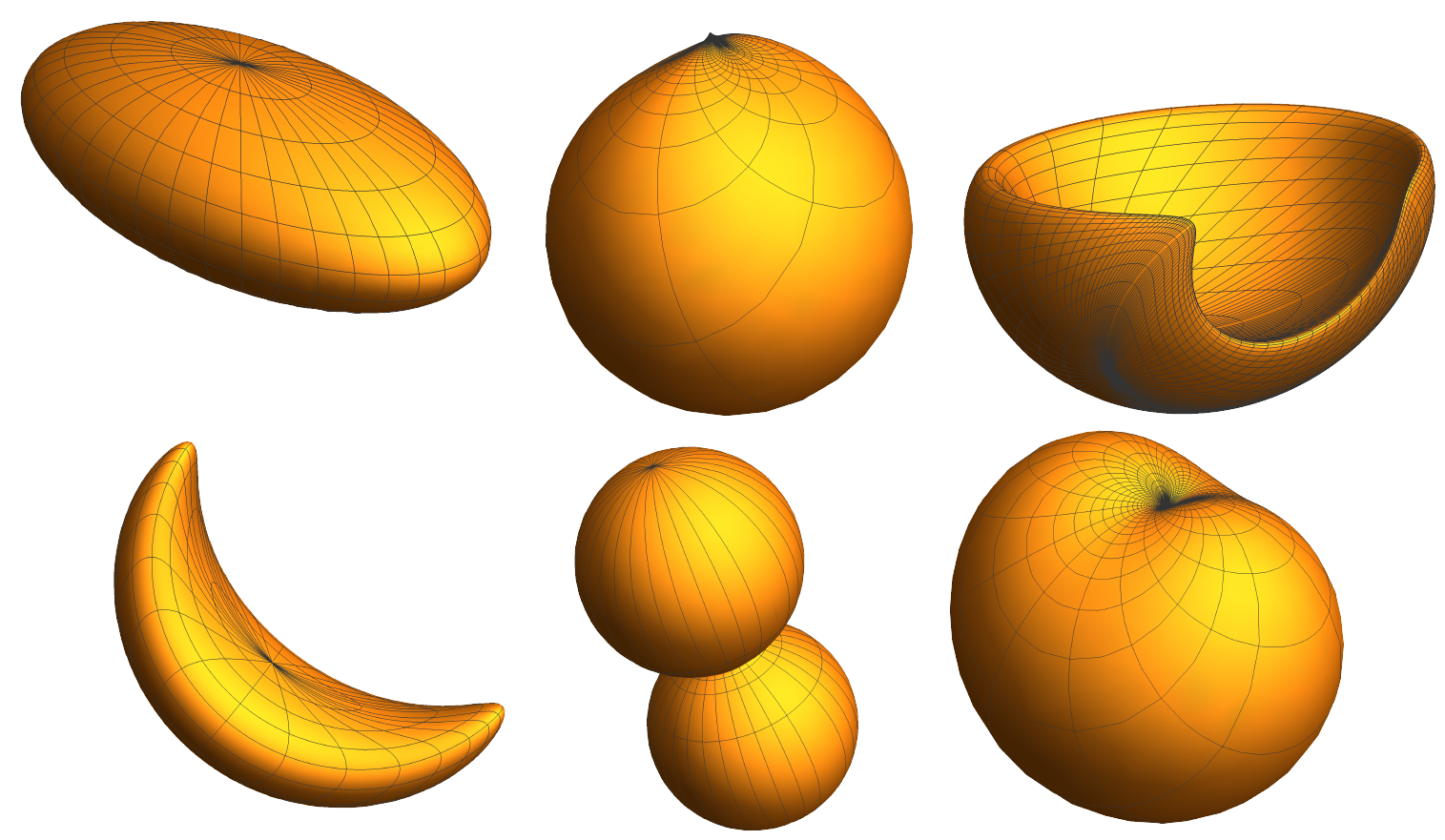}
\captionof{figure}{From left to right and top to bottom: The ellipsoid~$E_{(1,4,9)}$, $\Omega((1,4,9),1/2,(1.02,0,0))$, $\Omega((1,9,16^2),1/2,(-\frac{3}{10},-\frac{3}{10},-\frac{3}{10}))$, and~$\Omega((1,16,36),10,(0,0,-1.1))$. The last two figures are~$\Omega((1,4,9),1,(0,0,0))$ and~$\Omega((1,9,25),1/2,(-0.98,0,0))$, which are examples of \textit{unbounded} domains represented by the \textit{exterior} of the last two shapes.}\label{fig2}
\end{center}
\end{figure}

\appendix

\section{Asymptotic behavior of \texorpdfstring{$J_i^{(k)}$}{Jik}}

Recall that~$\mu$ is defined in~\eqref{constant-measure} with a diagonal matrix~$A$ with entries~$a_1=\ldots=a_{n-1}=1$ and~$a_n=a$. 
\begin{lemma}\label{asymptotic-ji}
Let~$n\geq 2$,~$k\in\{1,\ldots,n\}$, and
\begin{equation}\label{def:ji}
J_i^{(k)}=a_k^i\int_{\partial E_a} \theta_k^{2i} \mu(d\theta),\qquad i\in \N_0
\end{equation}
as in~\eqref{Ji}, where~$J_0:=J_0^{(1)}=\ldots=J_0^{(n)}$. Then
\begin{equation}\begin{split}\label{explicit-ji}
J_i^{(n)}&= a^{-1/2}\omega_{n-2}B\Big(i+\frac12,\frac{n-1}2\Big) \hf\Big(s+\frac{n}2, i+\frac12; i+\frac{n}2;1-\frac{1}{a}\Big)\qquad\text{and}\\
J_i^{(k)}=J_{i}^{(1)} &= a^{-1/2}\omega_{n-2}B\Big(i+\frac12,\frac{n-1}2\Big) \hf\Big(s+\frac{n}2, \frac12; i+\frac{n}2;1-\frac{1}{a}\Big)
\qquad\text{for~$k=1,\ldots,n-1$,}
\end{split}
\end{equation}
where~$\omega_{d}=\frac{2\pi^{(d+1)/2}}{\Gamma((d+1)/2)}=|\mathbb S^{d}|$ for~$d\in \N_0$. Moreover,~$\lim\limits_{a\uparrow \infty} \frac{J_i^{(n)}}{J_0}=1$ and,
\begin{enumerate}
	\item if~$s>i-\frac{1}{2}$, then
	\begin{align}\label{small-ji}
		\lim_{a\uparrow \infty} \frac{a^iJ_i^{(1)}}{J_0}=\frac{\Gamma(i+\frac{1}{2})\Gamma(\frac{1}{2}+s-i)}{\Gamma(\frac{1}{2})\Gamma(\frac{1}{2}+s)}=\prod_{k=0}^{i-1}\frac{1+2k}{2s-2k-1};
	\end{align}
	\item  If~$s\leq i-\frac{1}{2}$, then 
	\[
	\lim_{a\uparrow\infty} a^{ \frac{1}2}J_i^{(1)}=\omega_{n-2}\frac{B\Big(i+\frac12,\frac{n-1}2\Big)  B\Big(i-s-\frac{1}2,\frac{ 1}2\Big)}{B\Big(\frac12,\frac{n-1}2+i\Big)}
	\]
	and in particular~$\lim\limits_{a\uparrow \infty} \frac{a^{i}J_i^{(1)}}{J_0}=\infty$ and~$\lim\limits_{a\uparrow \infty} \frac{a^{i-j}J_i^{(1)}}{J_0}=0$ for~$j\in\{1,\ldots,i\}$ with~$s>i-j-\frac{1}{2}$.
\end{enumerate}
\end{lemma}
\begin{proof}
Let~$\theta=(\sin(\phi_{n-1})P_{n-2}(\phi'),a^{-1/2}\cos(\phi_{n-1}))$, with~$\phi_{n-1}\in(-\pi,\pi)$ and~$P_{n-2}(\phi')$ is the parametrization of~$\partial B_1^{n-1}(0)\cap\{x_n>0\}$, that is~$P_0\equiv 1$ and for~$n>2$,
\[
P_{n-2}=\Big(P_{n-3}(\phi_1,\ldots,\phi_{n-3})\sin(\phi_{n-2}),\cos(\phi_{n-2})\Big),\quad\text{~$\phi_k\in(0,\pi)$ for~$k=1,\ldots,n-2$}
\]
Then
\[
\det J_{\theta}^TJ_{\theta}=\cos^2(\phi_{n-1})+a^{-1}\sin^2(\phi_{n-1})
\]
for~$n=2$ and for~$n>2$ we have
\begin{align*}
\det J_{\theta}^TJ_{\theta}&=\det\left(\begin{array}{cc}\cos(\phi_{n-1}){P_{n-2}^T}(\phi') & -a^{-1/2}\sin(\phi_{n-1})\\ \sin(\phi_{n-1})J_{P_{n-2}}^T(\phi') & 0\end{array}\right)
\left(\begin{array}{cc}\cos(\phi_{n-1}){P_{n-2}}(\phi') & \sin(\phi_{n-1})J_{P_{n-2}}(\phi')\\ -a^{-1/2}\sin(\phi_{n-1}) & 0\end{array}\right)\\
&=\det\left(\begin{array}{cc}\cos^2(\phi_{n-1})+a^{-1}\sin^2(\phi_{n-1}) & 0\\ 0 & \sin^2(\phi_{n-1})J_{P_{n-2}}^T(\phi')J_{P_{n-2}}(\phi')\end{array}\right)\\
&=\big(\cos^2(\phi_{n-1})+a^{-1}\sin^2(\phi_{n-1})\big) \, \big(\sin^2(\phi_{n-1})\big)^{n-2}\det J_{P_{n-2}}^T(\phi')J_{P_{n-2}}(\phi')\\
&=\big(\cos^2(\phi_{n-1})+a^{-1}\sin^2(\phi_{n-1})\big) \, \big(\sin^2(\phi_{n-1})\big)^{n-2}\prod_{k=1}^{n-2}\sin^{2(k-1)}\phi_k.
\end{align*}
We begin with~$k=n$, where the above parametrization gives
\begin{align*}
& J_i^{(n)} =
\frac{\omega_{n-2}}{2}\int_{-\pi}^{\pi}\frac{\cos^{2i}(\phi_{n-1})\big(\cos^2(\phi_{n-1})+\frac{1}{a}\sin^2(\phi_{n-1})\big)^{1/2}\big(\sin^2(\phi_{n-1})\big)^{n/2-1}}{\big(\sin^2(\phi_{n-1})+\frac{1}{a}\cos^2(\phi_{n-1}))\big)^{s+n/2}\big(\sin^2(\phi_{n-1})+a\cos^2(\phi_{n-1})\big)^{1/2}} \;d\phi_{n-1} \\
&= \frac{\omega_{n-2}}{2a^{ 1/2}}\int_{-\pi}^{\pi}\frac{\cos^{2i}(\phi_{n-1})\big(1-\cos^2(\phi_{n-1})\big)^{n/2-1}}{\big(1-(1-\frac{1}{a})\cos^2(\phi_{n-1}))\big)^{s+n/2}} \,d\phi_{n-1} = \frac{2\omega_{n-2}}{a^{1/2}}\int_0^{\pi/2}\frac{\cos^{2i}(\phi_{n-1})\big(1-\cos^2(\phi_{n-1})\big)^{n/2-1}}{\big(1-(1-\frac{1}{a})\cos^2(\phi_{n-1}))\big)^{s+n/2}} \,d\phi_{n-1} 
\end{align*}
by symmetry. With the change of variable~$\phi_{n-1}=\arccos(t)$,~$\frac{d\phi_{n-1}}{dt}=-\frac{1}{\sqrt{1-t^2}}$ (and afterwards~$t^2=\tau$) it follows that
\begin{align*}
2\int_{0}^{\pi/2}&\frac{\cos^{2i}(\phi_{n-1})\big(1-\cos^2(\phi_{n-1})\big)^{n/2-1}}{\big(1-(1-\frac{1}{a})\cos^2(\phi_{n-1}))^{s+n/2}} \;d\phi_{n-1}
=2\int_{0}^{1}\frac{t^{2i}\big(1-t^2\big)^{(n-3)/2}}{\big(1-(1-\frac{1}{a})t^2)^{s+n/2}} \;dt \\
&=\int_{0}^{1}\tau^{i-1/2}(1-\tau)^{(n-3)/2}(1-(1-\frac{1}{a})\tau)^{-s-n/2} \;d\tau \\
&= B\Big(i+\frac12,\frac{n-1}2\Big) \hf\Big(s+\frac{n}2, i+\frac12; i+\frac{n}2;1-\frac{1}{a}\Big),
\end{align*}
where we have used the integral representation of the hypergeometric function~$\hf$. This proves~\eqref{explicit-ji} for~$k=n$. 

In the following, given two functions~$f$ and~$g$, we use notation~$f\sim g$ as~$a\uparrow\infty$, if~$\lim\limits_{a\uparrow\infty} \frac{f(a)}{g(a)}=1$. With the change of variable~$t=(a-1)\tau$ we have as~$a\uparrow\infty$
\begin{align}
& \,B\Big(i+\frac{1}{2},\frac{n-1}{2}\Big)\hf\Big(s+\frac{n}{2}, i+\frac{1}{2}; i+\frac{n}{2};1-\frac{1}{a}\Big) 
\sim \int_{1/2}^1 \tau^{i-1/2}(1-\tau)^{(n-3)/2}\bigg(1-\Big(1-\frac{1}{a}\Big)\tau\bigg)^{-s-n/2} \;d\tau \notag\\
&= \frac{a^{s+n/2}}{{(a-1)}^{i+1/2}}\int_{(a-1)/2}^{a-1} t^{i-1/2}\Big(1-\frac{t}{a-1}\Big)^{(n-3)/2}{(a-t)}^{-s-n/2} \;dt \notag\\
&=\frac{a^{s+n/2}}{{(a-1)}^{i+n/2-1}}\int_{(a-1)/2}^{a-1} t^{i-1/2}(a-1-t)^{(n-3)/2}{(a-t)}^{-s-n/2} \;dt \notag\\
&=\frac{a^{s+n/2}}{(a-1)^{i+n/2-1}}\int_{0}^{(a-1)/2} (a-1-t)^{i-1/2} t^{(n-3)/2}{(t+1)}^{-s-n/2} \;dt \notag\\
&=\frac{a^{s+n/2}}{(a-1)^{(n-1)/2}}\int_{0}^{(a-1)/2}\Big(1-\frac{t}{a-1}\Big)^{i-1/2} t^{(n-3)/2}{(t+1)}^{-s-n/2}  \;dt \notag\\
&\sim a^{s+1/2} \int_0^{\infty}t^{(n-3)/2}(t+1)^{-s-n/2}\;dt=a^{s+1/2}  B\Big(\frac{n-1}{2},s+\frac{1}2\Big).
\label{useful for J0}
\end{align}\label{useful for J02}
Note now that the asymptotic behavior of~$J_i$ follows from~\eqref{useful for J0} and it reads
\begin{align}
J_i^{(n)}\sim \omega_{n-2}B\Big(\frac{n-1}{2},s+\frac{1}2\Big)\,a^{s}
\qquad \text{as } a\uparrow\infty.
\end{align} 
so that~$\lim\limits_{a\uparrow\infty}\frac{J_i^{(n)}}{J_0}=1$ as claimed.

For~$k=1,\ldots,n-1$, by symmetry, it follows that~$J_i^{(k)}=J_i^{(1)}$. Moreover, with the above parametrization we have
\[
\theta_1=\prod_{k=1}^{n-1}\sin(\phi_k),\quad\text{with~$\phi_k\in(0,\pi)$},
\]
so that with a similar calculation as for~$k=n$ we have
\begin{align}
J_i^{(1)}&=\int_{-\pi}^{\pi}\frac{\big(\cos^2(\phi_{n-1})+\frac{1}{a}\sin^2(\phi_{n-1})\big)^{1/2}\big(\sin^2(\phi_{n-1})\big)^{n/2-1+i}}{\big(\sin^2(\phi_{n-1})+\frac{1}{a}\cos^2(\phi_{n-1}))\big)^{s+n/2}\big(\sin^2(\phi_{n-1})+a\cos^2(\phi_{n-1})\big)^{1/2}} \;d\phi_{n-1} \, \Bigg(\prod_{k=1}^{n-2}\int_{0}^{\pi}\sin^{k-1+2i}(\phi_k)\;d\phi_k\Bigg)\notag\\
&=\frac{4}{a^{1/2}}\int_0^{\pi/2}\frac{\big(1-\cos^2(\phi_{n-1})\big)^{n/2-1+i}}{\big(1-(1-\frac{1}{a})\cos^2(\phi_{n-1}))\big)^{s+n/2}} \,d\phi_{n-1} \, \Bigg(\prod_{k=1}^{n-2}\frac{\Gamma\big(\frac{1}{2}\big)\Gamma\big(i+\frac{k}2\big)}{\Gamma\big(i+\frac{k+1}2\big)}\Bigg)\notag\\
&=\frac{2\pi^{(n-2)/2}}{a^{1/2}}\int_{0}^{1}\tau^{-1/2}(1-\tau)^{(n-3)/2+i}\bigg(1-\Big(1-\frac{1}{a}\Big)\tau\bigg)^{-s-n/2} \;d\tau \, \Bigg(\prod_{k=1}^{n-2}\frac{ \Gamma\big(i+\frac{k}2\big)}{\Gamma(i+\frac{k+1}2\big)}\Bigg)\notag\\
&=\frac{2\pi^{(n-2)/2} }{a^{1/2}} B\Big(\frac12,\frac{n-1}2+i\Big) \hf\Big(s+\frac{n}2, \frac12; i+\frac{n}2;1-\frac{1}{a}\Big)\frac{\Gamma\big(i+\frac12\big)}{\Gamma\big(i+\frac{n-1}2\big)},\label{helpful-asymptotics}
\end{align}
from which~\eqref{explicit-ji} follows for~$k=1,\ldots,n-1$. Note that if~$s<i-\frac{1}{2}$, then, using again the integral representation of the hypergeometric function and the dominated convergence theorem, we have
\begin{align*}
& \lim_{a\uparrow\infty}\hf\Big(s+\frac{n}2, \frac12; i+\frac{n}2;1-\frac{1}{a}\Big) \ = \\
& =\ B\Big(\frac12,\frac{n-1}2+i\Big)^{-1} \lim_{a\uparrow\infty} \int_{0}^{1}{\tau}^{-1/2}\big(1-\tau\big)^{(n-3)/2+i}\bigg(1-\Big(1-\frac{1}{a}\Big)\tau\bigg)^{-s-n/2} \;d\tau \\
& =\ B\Big(\frac12,\frac{n-1}2+i\Big)^{-1}\int_{0}^{1}{\tau}^{-1/2}\big(1-\tau\big)^{i-s-3/2}  \;d\tau =\ B\Big(\frac12,\frac{n-1}2+i\Big)^{-1} B\Big(i-s-\frac{1}2,\frac{ 1}2\Big)
\end{align*}
by the integral representation of the beta function. Hence in this case
\begin{equation}\label{second-case-ji-part1}
\lim_{a\uparrow\infty}  a^{ \frac{1}2}J_i^{(1)}=\omega_{n-2}\frac{B\big(i+\frac12,\frac{n-1}2\big)  B\big(i-s-\frac{1}2,\frac12\big)}{B\big(\frac12,\frac{n-1}2+i\big)},
\end{equation}
which shows the first part in 2. If~$s>i-\frac12$ then with the change of variable~$t=(a-1)\tau$ we have from~\eqref{helpful-asymptotics} as~$a\uparrow\infty$
\begin{align}
a^iJ_i^{(1)}
&=\frac{2\pi^{(n-2)/2} }{a^{-i+1/2}}\frac{\Gamma\big(i+\frac12\big)}{\Gamma\big(i+\frac{n-1}2\big)} \, B\Big(\frac12,\frac{n-1}2+i\Big) 
\hf\Big(s+\frac{n}2, \frac12; i+\frac{n}2;1-\frac{1}{a}\Big)\notag\\
&\sim \frac{2\pi^{(n-2)/2}}{a^{-i+1/2}}\frac{\Gamma\big(i+\frac12\big)}{\Gamma\big(i+\frac{n-1}2\big)}
\int_{1/2}^{1}\tau^{-1/2}(1-\tau)^{(n-3)/2+i}\bigg(1-\Big(1-\frac{1}{a}\Big)\tau\bigg)^{-s-n/2} \;d\tau \notag\\
&=\frac{2\pi^{(n-2)/2}a^{s+(n-1)/2+i}}{(a-1)^{1/2}}\frac{\Gamma\big(i+\frac12\big)}{\Gamma\big(i+\frac{n-1}2\big)}
\int_{(a-1)/2}^{(a-1)}t^{-1/2}\Big(1-\frac{t}{a-1}\Big)^{(n-3)/2+i}(a-t)^{-s-n/2} \;dt\notag\\
&=\frac{2\pi^{(n-2)/2}a^{s+(n-1)/2+i}}{(a-1)^{n/2+i-1}}\frac{\Gamma\big(i+\frac12\big)}{\Gamma\big(i+\frac{n-1}2\big)}
\int_{(a-1)/2}^{(a-1)}t^{-1/2}(a-1-t)^{(n-3)/2+i}(a-t)^{-s-n/2} \;dt\notag\\
&=\frac{2\pi^{(n-2)/2}a^{s+(n-1)/2+i}}{(a-1)^{n/2+i-1}}\frac{\Gamma\big(i+\frac12\big)}{\Gamma\big(i+\frac{n-1}2\big)}
\int_{0}^{(a-1)/2}(a-1-t)^{-1/2}t^{(n-3)/2+i}(t+1)^{-s-n/2} \;dt\notag\\
&=\frac{2\pi^{(n-2)/2}a^{s+(n-1)/2+i}}{(a-1)^{(n-1)/2+i}}\frac{\Gamma\big(i+\frac12\big)}{\Gamma\big(i+\frac{n-1}2\big)}
\int_{0}^{(a-1)/2}\Big(1-\frac{t}{a-1}\Big)^{-1/2}t^{(n-3)/2+i}(t+1)^{-s-n/2} \;dt\notag\\
&\sim 2\pi^{(n-2)/2}a^{s}\frac{\Gamma\big(i+\frac12\big)}{\Gamma\big(i+\frac{n-1}2\big)}\int_{0}^{\infty} t^{(n-3)/2+i}(t+1)^{-s-n/2} \;dt\notag\\
&=2\pi^{(n-2)/2}a^{s}\frac{\Gamma\big(i+\frac12\big)}{\Gamma(i+\frac{n-1}2\big)} \, B\Big(i+\frac{n-1}2,s-i+\frac12\Big).\label{useful for J03}
\end{align}
Finally, we have with~\eqref{useful for J0}
\begin{align*}
\lim_{a\uparrow\infty}\frac{a^iJ_i^{(1)}}{J_0}
&=\frac{2\pi^{(n-2)/2} \Gamma\big(i+\frac12\big) B\big(i+\frac{n-1}2,s-i+\frac12\big)}{\omega_{n-2}\Gamma\big(i+\frac{n-1}2\big)\,B\big(\frac{n-1}{2},s+\frac{1}2\big)}\\
&= \frac{ \Gamma\big(i+\frac12\big) \Gamma\big(s-i+\frac12\big)}{\pi^{1/2} \Gamma\big(s+\frac12\big)}=\prod_{k=0}^{i-1}\frac{1+2k}{2s-2k-1}, 
\qquad \text{ if } s>i-\frac{1}{2}.
\end{align*}
as claimed in 1. 
	If instead~$s< i-\frac{1}2$, then by~\eqref{second-case-ji-part1} we have
	\begin{align*}
	\lim_{a\uparrow\infty }\frac{a^{i-j}J_i^{(1)}}{J_0}=
	\left\{\begin{aligned} 
	& +\infty & & \text{if }i-j>s+\frac12, \\
	& 0 & &\text{if }i-j<s+\frac12. 
	\end{aligned}\right.
	\end{align*}
	The case of~$s=i-\frac{1}{2}$ now follows similarly, noting that in this case~$a^{1/2}J_i=O(\ln(a))$ for~$a\uparrow\infty$.

\end{proof}

\begin{lemma}\label{identities-ik}
In the notations of Lemma~\ref{asymptotic-ji}, we have
\begin{align}
\int_{\partial E_a}\theta_i^2\,\theta_k^2 \; \mu(d\theta) & = \frac13 J_2^{(1)} & \text{for }i,k\in\{1,\ldots,n-1\},\ i\neq k, \label{ik} \\
a\int_{\partial E_a}\theta_i^2\,\theta_n^2 \; \mu(d\theta) & = J_1^{(1)}-\frac{n+1}{3}J_2^{(1)} & \text{for }i\in\{1,\ldots,n-1\}. \label{in}
\end{align}
\end{lemma}

\begin{proof}
The proof of~\eqref{ik} closely follows the computation in the proof of~\eqref{explicit-ji}.
Indeed, by symmetry,
\begin{align*}
\int_{\partial E_a}\theta_i^2\,\theta_k^2 \; \mu(d\theta) = \int_{\partial E_a}\theta_1^2\,\theta_2^2 \; \mu(d\theta)
\qquad\text{for }i,k\in\{1,\ldots,n-1\},\ i\neq k, 
\end{align*}
and with
\[
J:=\int_{-\pi}^{\pi}\frac{\big(\cos^2(\phi_{n-1})+\frac{1}{a}\sin^2(\phi_{n-1})\big)^{1/2}\big(\sin^2(\phi_{n-1})\big)^{n/2+1}}{\big(\sin^2(\phi_{n-1})+\frac{1}{a}\cos^2(\phi_{n-1}))\big)^{s+n/2}\big(\sin^2(\phi_{n-1})+a\cos^2(\phi_{n-1})\big)^{1/2}} \;d\phi_{n-1}\Bigg(\prod_{k=2}^{n-2}\int_{0}^{\pi}\sin^{k+3}(\phi_k)\;d\phi_k\Bigg),
\]
observing that~$\theta_2=\cos(\phi_1)\prod_{k=2}^{n-1}\sin(\phi_k)$, we have
\begin{multline*}
\int_{\partial E_a}\theta_1^2\,\theta_2^2 \; \mu(d\theta) = J\int_{0}^{\pi}\sin^{2}(\phi_1)\,\cos^2(\phi_1)\;d\phi_1 =J\Big(\int_{0}^{\pi}\sin^{2}(\phi_1)\;d\phi_1-\int_{0}^{\pi}\sin^{4}(\phi_1)\;d\phi_1\Big)= \\
=J_2^{(1)}\Big(\frac{\frac\pi2}{\frac38\pi}-1\Big)=\frac13 J_2^{(1)}.
\end{multline*}
For the proof of~\eqref{in} we proceed as follows using again the symmetry:
\begin{align*}
a\int_{\partial E_a}\theta_i^2\,\theta_n^2 \; \mu(d\theta) =
a\int_{\partial E_a}\theta_1^2\,\theta_n^2 \; \mu(d\theta) =
\int_{\partial E_a}\theta_1^2\,\Big(1-\sum_{i=1}^{n-1}\theta_i^2\Big) \; \mu(d\theta)  =
J_1^{(1)}-J_2^{(1)}-\frac{n-2}{3}J_2^{(1)},
\end{align*}
where we have used~\eqref{ik} in the last identity.
\end{proof}

\section*{Acknowledgments}
We thank the anonymous referee for the careful reading of the manuscript and for several helpful comments and suggestions. 


\end{document}